\documentclass[leqno, letterpaper, 11pt]{article}
\usepackage{amssymb}
\usepackage{amsmath}
\usepackage{amsthm}
\usepackage{tikz}
\usepackage{fullpage}
\usepackage{mathrsfs}
\usepackage{subfigure}
\usepackage{verbatim}
\usepackage{enumitem}
\usepackage{bbm}
\usepackage{lipsum}

\renewcommand{\P}{\mathbb{P}}

     \newcommand{\floor}[1]{\left\lfloor#1\right \rfloor}

\newcommand{\N}{\mathbb{N}}

     \newcommand{\abs}[1]{\left|#1\right|}

 \newcommand{\df}[1]{\emph{#1}}

\def\Z{\mathbb{Z}}

\newcommand{\bR}{{\mathbb R}}
\newcommand{\bZ}{{\mathbb Z}}

\newcommand{\bT}{{\mathbb T}}
\newcommand{\cC}{{\mathcal C}}

\newcommand{\cO}{\mathcal O}

\newcommand{\xk}{\{x\}\times K_n^2}
\newcommand{\p}{\mathsf{p}}
\newcommand{\q}{\mathsf{q}}
\newcommand{\nonbrs}{\text{\tt{Nbrs}}}

\newcommand{\eps}{\epsilon}

\newcommand{\ind}[1]{\mathbbm 1(#1)}
     
\numberwithin{equation}{section}
     
\newtheorem{theorem}{Theorem}[section]

\newtheorem{lemma}[theorem]{Lemma}

\newtheorem{prop}[theorem]{Proposition}

\newtheorem{question}[theorem]{Question}

\newcommand{\prob}[1]{\mathbb{P}\left(#1\right)}
\newcommand{\probsub}[2]{\mathbb{P}_{#1}\left(#2\right)}

\newcommand\blfootnote[1]{%
  \begingroup
  \renewcommand\thefootnote{}\footnote{#1}%
  \addtocounter{footnote}{-1}%
  \endgroup
}

\begin{title}
{Bootstrap percolation on the product\\
of the two-dimensional lattice with a Hamming square}
\end{title}

\author{Janko Gravner\\
Department of Mathematics\\University of California\\Davis, CA 95616\\{\tt gravner{@}math.ucdavis.edu}
\and David Sivakoff\\
Departments of Statistics and Mathematics\\
Ohio State University\\ Columbus, OH 43210\\
{\tt dsivakoff{@}stat.osu.edu}}

\begin{document}


%
%


\maketitle
\begin{abstract} 
Bootstrap percolation on a graph is a deterministic 
process that iteratively enlarges a set of occupied sites by adjoining points with at least $\theta$ 
occupied neighbors. The initially occupied
set is random, given by a uniform product measure with a low
density $p$. Our main focus is on this process on the product graph $\bZ^2\times K_n^2$, where $K_n$ is a complete graph. 
We investigate how  $p$ 
scales with $n$ so that a typical site is eventually occupied. 
Under critical 
scaling, the dynamics with even $\theta$ exhibits a sharp phase transition, while odd $\theta$ yields a gradual percolation 
transition. We also establish a gradual transition 
for bootstrap percolation on $\bZ^2\times K_n$. 
The main tool is heterogeneous bootstrap percolation 
on $\bZ^2$. 
\end{abstract}



\section{Introduction}

Spread of signals --- information, say, or infection ---
on graphs with community structure has attracted interest in the mathematical literature 
recently~\cite{Schi, BL,  Lal, LZ, Siv}. 
The idea is that any single community 
is densely connected, while the connections between communities
are much more sparse. This naturally leads to 
multiscale phenomena, as the spread of the signal within 
a community is much faster then between different 
communities. Often, communities are modeled as cliques,
i.e., the intra-community graph is complete, but in other cases 
some close-knit structure is assumed. By contrast, the inter-community 
graph may, for example, impose spatial proximity as a precondition for connectivity. See \cite{Sil} 
for an applications-oriented recent survey. 

The principal 
graph under study in this paper is $G=\bZ^2\times K_n^2$, 
the Cartesian 
product between the lattice $\bZ^2$ and two copies of the 
complete graph $K_n$ on $n$ points.
Thus ``community'' 
consists of ``individuals'' determined by two characteristics, 
and two individuals within the community only communicate 
if they have one of the characteristics in common. Between 
the communities, communication is between like 
individuals that are also neighbors in the lattice. For comparison, we also 
address the case where each community is a clique, that is, the 
graph
$\bZ^2\times K_n$. 

The 
particular dynamics we use for spread of signals 
is \emph{bootstrap percolation} with integer \emph{threshold} parameter 
$\theta\ge 1$. In this very simple deterministic process, 
one starts with an initial configuration $\omega_0$ 
of $0$s (or \emph{empty} sites) and $1$s 
(or \emph{occupied} sites) on vertices of $G$, and iteratively 
enlarges
the set of occupied sites in discrete time as follows. 
Assume $\omega_t$ is given for some $t\ge 0$, and fix a vertex 
$v$ of $G$. If $\omega_t(v)=1$, then $\omega_{t+1}(v)=1$. 
If $\omega_t(v)=0$, and $v$ has $\theta$ or more neighboring 
vertices $v'$ with $\omega_t(v')=1$, then $\omega_{t+1}(v)=1$; 
otherwise $\omega_{t+1}(v)=0$. We will typically identify 
the configuration $\omega_t$ with the set of its 
occupied sites $\{v:\omega_t(v)=1\}$. Thus $\omega_t$ 
increases to the set $\omega_\infty=\cup_{t\ge 0}\omega_t$ of  eventually occupied vertices.

 As is typical, we assume 
that the initial state $\omega_0$ is a uniform 
product measure with some small density $p\in (0,1)$. This makes 
the set $\omega_\infty$ random as well, and it is natural 
to ask how to choose $p$ to make 
$\omega_\infty$ large, i.e., to make the initially sparse
signal widespread. Observe that, 
if $\theta\ge 3$, $\omega_\infty$ cannot comprise
{\it all\/} vertices of $G$ with nonzero probability for 
any $p<1$, 
as a block of neighboring empty copies of $K_n^2$ 
(e.g, $\{(0,0),(0,1),(1,0),(1,1)\}\times K_n^2$) cannot 
be invaded by occupied sites, and the infinite lattice will 
contain such a block with probability $1$. We therefore
ask a weaker question:  
how large should $p$ be, in terms of $n$, so that 
$\omega_\infty$ comprises a substantial proportion of 
points? That is, we are interested in 
the size of the \emph{final density}, $\probsub{p}{v_0\in \omega_\infty}$, 
which is independent of $v_0\in G$ by vertex-transitivity of $G$.

Bootstrap percolation was introduced on trees in \cite{CLR}, 
but it has received
by far the most attention on lattices $\bZ^d$. 
In this case,  
$\probsub{p}{\omega_\infty=\bZ^d}=1$, as proved in \cite{vE}
for $d=2$ and in \cite{Scho} for $d\ge 3$. Many deep and 
surprising results originated from the study of 
metastability  properties  of  the  model  on  finite  regions  (see e.g. [AL, Hol, BBDM, GHM]).  We refer to the recent survey [Mor] for a comprehensive review. 

Study of bootstrap percolation and related dynamics on graphs with long-range connectivity is a more recent undertaking \cite{GHPS, Sli, GSS, BBLN, GPS} 
and has a fundamentally different flavor: while on sparse 
graphs, the dominant mechanism is formation of small 
nuclei that are likely to grow indefinitely, the relevant 
events in 
densely connected graphs tend to depend on 
the configuration on the whole space. It is therefore 
tempting to consider graphs that combine aspects of 
both, and we continue here our work started in \cite{GS}. 

As already remarked, $\omega_\infty$ cannot cover all 
vertices of our graph $G$ due to the presence of 
local configurations of 
sparsely occupied copies of Hamming squares, $K_n^2$. 
Other copies, of course, have higher initial occupation, get fully occupied and 
spread their occupation to the neighboring squares. Thus we have
a competition between densely occupied copies 
of $K_n^2$ that act as nuclei, and sparsely occupied ones 
that function as obstacles to growth. This invites 
comparison with polluted bootstrap percolation
\cite{GM, GH, GHS} on $\bZ^2$, which is indeed the 
main source of our tools. However, 
by contrast with the model in the 
cited papers, which has only three
states (empty and occupied sites, and permanent obstacles),
the dynamics that arise from our process has more types 
corresponding to all possible thresholds ($0, 1,2,3,4, 5$) 
that different 
sites in $\bZ^2$ require to become occupied. Moreover, we
need different variants for 
the case $\theta=3$ and the graph $\bZ^2\times K_n$.
We call 
these comparison 
dynamics \emph{heterogeneous bootstrap percolation}. We also
encounter a technical difficulty in the form of 
correlations in the initial state, which are handled
by coupling and 
other related perturbation methods.  

After its introduction in \cite{GM}, the basic polluted 
version of heterogeneous bootstrap percolation 
was further analyzed in \cite{GH, GHS}; it is the 
recent techniques developed in these two papers 
that will be useful to us.  Related models include processes  on a complete graph  with  excluded  edges \cite{JLTV}, Glauber dynamics with “frozen” vertices \cite{DEKNS}, dynamics on complex networks with “damaged” vertices \cite{BDGM1, BDGM2}, 
and on inhomogeneous geometric random graphs \cite{KL}.

Our main results determine a critical scaling for 
prevalent occupation on $\bZ^2\times K_n^2$: we exhibit functions $f_\theta(n)$ so that, 
when $p=af_\theta(n)$, the limit as $n\to\infty$ of 
the final density
$\probsub{p}{v_0\in \omega_\infty}$ is low for small $a$ and 
high for large $a$. In fact, for all $\theta$, this limit vanishes for $a<a_c$, where $a_c=a_c(\theta)$ is a critical 
value that we are able to identify (and in fact compute 
explicitly for even $\theta$).  
The behavior for $a>a_c$ is however not 
the same for all $\theta$: if $\theta$ is even, the limit is 
$1$, while if $\theta$ is odd the final density is bounded away from $1$ 
for any finite $a$ and 
only approaches $1$ as $a\to\infty$. 
We already encountered the non-intuitive qualitative difference
between odd and even $\theta$ in our 
earlier work \cite{GS}, in which the lattice factor was one
dimensional.  This, and the connection with heterogeneous 
bootstrap percolation, are the most inviting features of our present model.
 
We now proceed to formal statements of our results. We first remark 
that for $\theta\le 2$ we have no 
obstacles and 
$\probsub{p}{\omega_\infty\equiv 1}=1$ for any $p>0$ by 
standard bootstrap percolation arguments \cite{vE, Scho}; 
therefore, we assume that $\theta\ge 3$ throughout the paper.  As we have 
so far, we denote by $v_0$ an arbitrary fixed vertex of the 
graph in question, and we use the notation $\mathbf{0}=(0,0)$ 
for the origin in $\bZ^2$. We begin with our main result 
for even thresholds.

\begin{theorem}\label{sharp transition even}
Consider bootstrap percolation on $\bZ^2\times K_n^2$ with 
threshold $\theta = 2\ell+2$, for some $\ell\ge 1$. Assume that
\begin{equation}\label{form-p}
p=a\cdot \frac{(\log n)^{1/\ell}}{n^{1+1/\ell}}, 
\end{equation}
for some $a>0$.  

If $a^\ell < 2(\ell-1)!$, then
\begin{equation}\label{subcritical-even}
\probsub{p}{v_0\in \omega_\infty} 
=
n^{-2/\ell+o(1)}
\qquad \text{as $n\to\infty$}.
\end{equation}
Conversely, if $a^\ell \ge 2(\ell-1)!$, then
\begin{equation}\label{supercritical-even}
\probsub{p}{\{\mathbf{0}\}\times K_n^2 \subset \omega_\infty} \to 1 \qquad \text{as $n\to\infty$}.
\end{equation}
Moreover, if $a^\ell > 2(\ell-1)!$ , then
\begin{equation}\label{supercritical-even1}
\probsub{p}{\{\mathbf{0}\}\times K_n^2 \not\subset \omega_\infty}=
\begin{cases}
n^{4/\ell-4a^\ell/\ell!+o(1)} &\ell\ge 2\\
n^{-2a+o(1)} &\ell=1
\end{cases}\qquad \text{as $n\to\infty$},
\end{equation}
and $\probsub{p}{\omega_0=\omega_\infty\text{ on }\{\mathbf{0}\}\times K_n^2}$ satisfies the same asymptotics.
\end{theorem}

Our results for odd thresholds are somewhat less precise, 
but suffice to provide the announced distinction from 
even $\theta$.

\begin{theorem}\label{odd-theta-thm}
Consider bootstrap percolation on $\bZ^2\times K_n^2$ with 
threshold $\theta=2\ell+1$, for some $\ell\ge 1$. Assume that 
\begin{equation}\label{form-p-odd}
p = \frac{a}{n^{1+1/\ell}},
\end{equation}
for some $a>0$.  

There exists a critical value $a_c=a_c(\ell)\in (0,\infty)$ so that the following holds. If $a<a_c$, then
\begin{equation}\label{odd-theta-sub}
\probsub{p}{v_0\in \omega_\infty} \to 0 \qquad \text{as $n\to\infty$}.
\end{equation}
Conversely, if $a>a_c$, then
\begin{equation}\label{odd-theta-super}
0<\liminf_{n\to\infty} \probsub{p}{\{\mathbf{0}\}\times K_n^2 \subset \omega_\infty} \le \limsup_{n\to\infty} \probsub{p}{v_0\in \omega_\infty}< 1.
\end{equation}
Furthermore, 
\begin{equation}\label{odd-theta-super1}\liminf_{n\to\infty} \probsub{p}{\{\mathbf{0}\}\times K_n^2 \subset \omega_\infty}\to 1\text{ as }a\to\infty.
\end{equation}
\end{theorem}

Finally, we state our result for the case of clique community, 
in which there is no difference between odd and even $\theta$
and no phase transition as in Theorems~\ref{sharp transition even} and~\ref{odd-theta-thm}. 

\begin{theorem}\label{z2xK theorem}
Consider bootstrap percolation on $\bZ^2\times K_n$ 
with threshold $\theta\ge 3$. Assume that $p=a/n$ for 
some $a\in (0,\infty)$. Then both 
$\liminf_n\probsub{p}{\omega_\infty(v_0)=1}$ and 
$\limsup_n\probsub{p}{\omega_\infty(v_0)=1}$ are in $(0,1)$ 
and converge to $0$ (resp.~$1$) as $a\to 0$ (resp.~$a\to\infty)$. 
If $\theta\ge 14$, then $\lim_n\probsub{p}{\omega_\infty(v_0)=1}$ exists 
and is continuous in $a$. 
\end{theorem}

A similar result to Theorem~\ref{z2xK theorem} holds for $\bZ^d\times K_n$ for all $d\ge 3$, but extension of our results 
to $\bZ^d\times K_n^2$ is much more challenging (see 
Section 7 on open problems). 

The organization of the rest of the paper is as follows. In the 
Section 2, we state the necessary properties of random 
subsets of a single copy $\{x\}\times K_n^2$ of a Hamming square, and in this we mostly review the results from \cite{GHPS, GS}. 
We also introduce the basic version of the heterogeneous 
bootstrap percolation; two other variants are used in 
Subsection 5.4 and Section 6. In Section 3, we prove the
subcritical rate (\ref{subcritical-even}), by utilizing the
connection with polluted bootstrap percolation \cite{GM,GHS}. 
Our argument closely follows that of \cite{GHS}, but we
give a substantial amount of details due to the differences
in the assumptions and conclusions. In Section 4, we focus 
on the supercritical part of Theorem~\ref{sharp transition even}, 
which is handled by the method from \cite{GM}, and then involves finding 
the most likely configuration that prevents occupation 
from spreading inwards from a circuit of fully 
occupied copies of the Hamming square. Section 5 contains 
the proof Theorem~\ref{odd-theta-thm}, in which we 
characterize $a_c$ through the limiting dynamics (as
$n\to\infty$), which can be appropriately coupled to the 
dynamics for finite $n$. A different limiting dynamics is
similarly used in Section 6, which is devoted to the proof 
of Theorem~\ref{z2xK theorem}. We conclude with a list of 
open problems in Section 7.



\section{Preliminaries}

\subsection{Copies of Hamming squares}
Fix an initial state 
$\omega_0$ for our bootstrap dynamics on $\bZ^2\times K_n^2$.  
For a set $A\subset \bZ^2\times K_n^2$, the dynamics 
\emph{restricted to} $A$ uses the bootstrap rule on the subgraph 
induced by $A$, with the initial state $\omega_0$ on $A$.
As in \cite{GS}, we call a copy $\{x\}\times K_n^2$, $x\in\bZ^2$:
\begin{itemize}
\item \emph{internally spanned at threshold $r$} ($r$-\emph{IS})
if the bootstrap dynamics with threshold $r$, restricted to $\{x\}\times K_n^2$,
eventually results in full occupation of $\{x\}\times K_n^2$;
\item \emph{internally inert at threshold $r$} ($r$-\emph{II})
if the bootstrap dynamics with threshold $r$, restricted to $\{x\}\times K_n^2$,
never changes the state of any vertex in $\{x\}\times K_n^2$;
and
\item \emph{inert} at threshold $r$ ($r$-\emph{inert}) if 
the (unrestricted) bootstrap dynamics with threshold $r$ 
does not occupy 
any point in $\{x\}\times K_n^2$ in the first time step. 
\end{itemize}

In the rest of this subsection, we mostly 
summarize the results from \cite{GS} and \cite{GHPS}.
We begin with the results for even $\theta$, which 
were essentially proved in \cite{GS}.

\begin{lemma}\label{internal-results-even}
Assume that $p$ is given by (\ref{form-p}). 
\begin{enumerate}
\item If $\ell\ge 1$, then 
$$
\probsub{p}{\text{$K_n^2$ is not $(2\ell-2)$-IS}}=\cO(n^{-L}),
$$
for any constant $L>0$.
\item 
If $\ell\ge 2$, then 
$$
\probsub{p}{\text{$K_n^2$ is not $(2\ell-1)$-IS}}\sim \probsub{p}{\text{$K_n^2$ is  $(2\ell-1)$-II}}\sim n^{-2a^\ell/\ell!}.
$$
and for $\ell = 1$ we have
$$
\probsub{p}{\text{$K_n^2$ is not $1$-IS}}=
\probsub{p}{\text{$K_n^2$ is $1$-II}}\sim \frac{1}{n^{a}}.
$$
\item If $\ell\ge 2$, then 
$$
\probsub{p}{\text{$K_n^2$ is not $(2\ell)$-IS}}\sim 
\probsub{p}{\text{$K_n^2$ is $(2\ell)$-II}}\sim 2n^{-a^\ell/\ell!},
$$
and for $\ell = 1$ we have
$$
\probsub{p}{\text{$K_n^2$ is not $2$-IS}}\sim 
\probsub{p}{\text{$K_n^2$ is $2$-II}}\sim \frac{a\log n}{n^{a}}.
$$
\item If $\ell\ge 1$, then 
$$
\probsub{p}{\text{$K_n^2$ is $(2\ell+1)$-IS}}\sim 
\probsub{p}{\text{$K_n^2$ is not $(2\ell+1)$-II}}\sim 
\frac {2a^{\ell+1}}{(\ell+1)!}\cdot \frac{(\log n)^{1+1/\ell}}{n^{1/\ell}}.
$$
\item If $\ell\ge 1$, then 
$$
\probsub{p}{\text{$K_n^2$ is $(2\ell+2)$-IS}}\sim 
\probsub{p}{\text{$K_n^2$ is not $(2\ell+2)$-II}}\sim 
\left(\frac {a^{\ell+1}}{(\ell+1)!}\right)^2\cdot \frac{(\log n)^{2+2/\ell}}{n^{2/\ell}}.
$$
\end{enumerate}
\end{lemma}

\begin{proof}
Statements 1 through 4 are Lemmas 3.6, 3.3, 3.4 and 3.5 in~\cite{GS}, and the proof of the last statement is similar to the proof of the 4th, so we omit it.
\end{proof}

The next lemma compares probabilities for inertness and 
internal inertness for $\ell\ge 2$. 

\begin{lemma}\label{inertness-even}
Assume $\theta=2\ell+2$, $\ell\ge 2$, 
and $p$ is given by~\eqref{form-p}. If $\frac{a^\ell}{\ell!} < 1$, then for any $x\in \bZ^2$
\begin{equation*}
\begin{aligned}
\probsub{p}{\text{$\{x\}\times K_n^2$ is $(\theta-2)$-inert}} &\sim \probsub{p}{\text{$K_n^2$ is $(\theta-2)$-II}} \sim 2n^{-a^\ell/\ell!},\\
\probsub{p}{\text{$\{x\}\times K_n^2$ is not $(\theta-1)$-inert}} &\sim \probsub{p}{\text{$K_n^2$ is not $(\theta-1)$-II}} \sim \frac {2a^{\ell+1}}{(\ell+1)!}\cdot \frac{(\log n)^{1+1/\ell}}{n^{1/\ell}},\\
\probsub{p}{\text{$\{x\}\times K_n^2$ is not $\theta$-inert}} &\sim \probsub{p}{\text{$K_n^2$ is not $\theta$-II}} \sim \left(\frac {a^{\ell+1}}{(\ell+1)!}\right)^2\cdot \frac{(\log n)^{2+2/\ell}}{n^{2/\ell}}.
\end{aligned}
\end{equation*}
\end{lemma}

\begin{proof} Fix an $r=0,1,2$. Then the probability that any fixed copy of $K_n^2$ has a site with exactly $k\ge 1$ occupied $\bZ^2$-neighbors and at least $\theta-r-k$ occupied $K_n^2$-neighbors is
$$
\cO(n^2 p^k (np)^{\theta-r-k}) = \cO(n^{-k - (2-r)/\ell} (\log n)^{(2\ell+2-r)/\ell}).
$$
Therefore, 
\begin{equation*}
\probsub{p}{\text{$\{x\}\times K_n^2$ is $(\theta-r)$-II but not $(\theta-r)$-inert}} = n^{-1 - (2-r)/\ell+o(1)}.
\end{equation*}
The rest follows from Lemma~\ref{internal-results-even} parts 3, 4 and 5 and the assumptions put on $a$ and $\ell$.
\end{proof}

We need a slightly more involved argument for $\ell=1$.

\begin{lemma}\label{theta-4-2-inert}
Assume $\theta=4$ and $
p=a\frac{\log n}{n^2}.
$
We have, 
$$
\probsub{p}{\{x\}\times K_n^2\text{ is $2$-inert}}\ge 
a\frac{\log n}{n^a}(1-o(1))
$$
\end{lemma}

\begin{proof}
Let $G_1$ be the event that $\{x\}\times K_n^2$ contains 
at least two occupied vertices, and $G_2$ the event that 
a point in $\{x\}\times K_n^2$ has both an occupied 
$\bZ^2$-neighbor and an occupied $K_n$-neighbor. Note that 
these are increasing events and that 
$$
\{\{x\}\times K_n^2\text{ is not $2$-inert}\}\subset G_1\cup G_2.
$$
Therefore, by FKG inequality, 
$$
\probsub{p}{\{x\}\times K_n^2\text{ is not $2$-inert}}\le 
\probsub{p}{G_1}+\probsub{p}{G_2}-\probsub{p}{G_1}\probsub{p}{G_2},
$$
and so
$$
\probsub{p}{\{x\}\times K_n^2\text{ is $2$-inert}}\ge 
\probsub{p}{G_1^c}-\probsub{p}{G_1^c}\probsub{p}{G_2}.
$$
Finally, we use that $\probsub{p}{G_1^c}\sim a\frac{\log n}{n^a}$ 
and $\probsub{p}{G_2}\le 8n^3p^2=\cO(\log n/n)$.  
\end{proof}

We proceed with the analogous results for odd $\theta$, 
which mostly follow from \cite{GHPS}, and we again 
omit the detailed proofs.

\begin{lemma}\label{internal-results-odd}
Assume that $p$ is given by (\ref{form-p-odd}). 
\begin{enumerate}
\item If $\ell\ge 1$, then 
$$
\probsub{p}{\text{$K_n^2$ is not $(2\ell-2)$-IS}}=\cO(n^{-L}),
$$
for any constant $L>0$.
\item 
If $\ell\ge 2$, then 
$$
\probsub{p}{\text{$K_n^2$ is not $(2\ell-1)$-IS}}\sim \probsub{p}{\text{$K_n^2$ is  $(2\ell-1)$-II}}\sim \exp\left[-\frac{2 a^\ell}{\ell!}\right].
$$
\item If $\ell\ge 2$, then 
$$
\probsub{p}{\text{$K_n^2$ is  $(2\ell)$-IS}}\sim 
\probsub{p}{\text{$K_n^2$ is not $(2\ell)$-II}}\sim \left(1 - e^{-a^{\ell}\slash{\ell!}}\right)^2.
$$
\item 
If $\ell\ge 1$, then 
$$
\probsub{p}{\text{$K_n^2$ is not $(2\ell+1)$-II}}\sim 
2\cdot \frac{a^{\ell+1}}{(\ell+1)!} \cdot \left(1 - e^{-a^{\ell}\slash{\ell!}}\right)\cdot \frac{1}{n^{1/\ell}},
$$
and
$$
\probsub{p}{\text{$K_n^2$ is $(2\ell+1)$-IS}}\sim 2\cdot \frac{a^{\ell+1}}{(\ell+1)!} \cdot \left(1 - e^{-a^{\ell}\slash{\ell!}}\right)^2 \cdot \frac{1}{n^{1/\ell}}.
$$
\end{enumerate}
\end{lemma}

\begin{proof}
Parts 2 and 3 follow from Theorem 2.1 in \cite{GHPS}. Part 1 is 
proved in the same fashion as Lemma 3.6 in \cite{GS}. The proof of part 4 is similar to the proof of parts 2 and 3 and is 
omitted; in fact, we only need in our arguments in Section 5 that the two probabilities
are positive for all $n$ and go to $0$ as $n\to\infty$, which is 
very easy to show. 
\end{proof}


We conclude with an analogue of Lemma~\ref{inertness-even}.

\begin{lemma}\label{inertness-odd}
Assume that  $\theta = 2\ell+1$, $\ell\ge 1$, and that 
$p$ is given by~\eqref{form-p-odd}. Fix an 
$x\in \bZ^2$. Then, for $\ell\ge 2$, 
\begin{equation*}
\probsub{p}{\text{$\{x\}\times K_n^2$ is 
$(\theta-2)$-II but not $(\theta-2)$-inert}}=\cO(n^{-1})
\end{equation*}
and, for $\ell\ge 1$,
\begin{equation*}
\begin{aligned}
&\probsub{p}{\text{$\{x\}\times K_n^2$ is 
$(\theta-1)$-II but not $(\theta-1)$-inert}}=\cO(n^{-1}),\\
&\probsub{p}{\text{$\{x\}\times K_n^2$ is 
$\theta$-II but not $\theta$-inert}}=\cO(n^{-1-1/\ell}).\\
\end{aligned}
\end{equation*}
\end{lemma}
\begin{proof} Observe that, for
$r\in \{0, 1, 2\}$, 
the probability that any fixed copy of $K_n^2$ has a site with exactly $k\ge 1$ occupied $\bZ^2$-neighbors and at least $\theta-r-k$ occupied $K_n^2$-neighbors is
$$
\cO(n^2 p^k (np)^{\theta-r-k}) = \cO(n^{-k + (r-1)/\ell}), 
$$
and the desired estimates follow.  
\end{proof}

\subsection{Heterogeneous bootstrap percolation}

We now introduce a comparison bootstrap dynamics $\xi_t$
on $\bZ^2$, which is a generalization of polluted 
bootstrap percolation introduced in \cite{GM}. 
We assume that $\xi_t\in \{0,1,2,3,4,5\}^{\bZ^2}$, $t\in \bZ_+$, and that $\xi_0$ is given. 
The rules mandate that a state can only change to $0$ by contact with sufficient number of 
$0$s. More precisely, if $Z_t(x)$ is the cardinality of $\{y: y\sim x\text{ and }\xi_t(y)=0\}$, where $x \sim y$ means that $x$ and $y$ are nearest neighbors in $\bZ^2$,
then
$$
\xi_{t+1}(x)=
\begin{cases}
0 & Z_t(x)\ge \xi_t(x)\\
\xi_t(x) &\text{otherwise.}
\end{cases}
$$
If $\xi_0\in \{0,2\}^{\bZ^2}$, this is the usual threshold-2 bootstrap percolation. Adding $1$s adds
sites which need to be ``switched on'' by neighboring $0$s. Finally, $3$s, $4$s and $5s$ act like 
``obstacles,'' which prevent the spread of $0$s at sufficient density.

The next two lemmas establish upper and lower-bounding couplings between $\xi_t$ and $\omega_t$. Their proofs are similar, so we only provide details for the second one.
\begin{lemma}\label{omega lower bound}
Assume $\xi_0(x)=0$ whenever the Hamming plane $\{x\}\times K_n^2$ is $\theta$-IS; $\xi_0(x)=k\in\{1,2,3, 4\}$ whenever $\{x\}\times K_n^2$ is $(\theta-k)$-IS, 
but is not $(\theta-k+1)$-IS; and that $\xi_0(x)=5$ if $\{x\}\times K_n^2$ is not $(\theta-4)$-IS.  Then 
$$
\bigcup\{\{x\}\times K_n^2 : \xi_\infty(x)=0\}\subset\omega_\infty.
$$ 
\end{lemma}

\begin{lemma}\label{omega upper bound}
Assume $\xi_0(x)=0$ whenever the Hamming plane $\{x\}\times K_n^2$ is not $\theta$-inert; that $\xi_0(x)=k\in \{1,2,3,4\}$ whenever $\{x\}\times K_n^2$ is not $(\theta-k)$-inert,
but is $(\theta-k+1)$-inert; and that $\xi_0(x)=5$ if $\{x\}\times K_n^2$ is $(\theta-4)$-inert. Then 
$$
\omega_\infty\subset\bigcup\{\{x\}\times K_n^2: \xi_\infty(x)=0\}\cup \omega_0.
$$
\end{lemma}

\begin{proof}
We will prove the following stronger statement by induction. We claim that for every $t\ge 0$,
\begin{equation} \label{eqn:xi-comparison}
\omega_t \subset\bigcup\{\{x\}\times K_n^2: \xi_t(x)=0\}\cup \omega_0.
\end{equation}
Suppose that (\ref{eqn:xi-comparison}) holds through time $t-1\ge 0$, and let $x\in \bZ^2$ be a point such that $\xi_{t}(x)\neq 0$. Suppose $x$ has exactly $k$ neighbors $y\in \bZ^2$ with $\xi_{t-1}(y) = 0$. Therefore, $\xi_0(x) \ge k+1$, so $\xk$ is $(\theta-k)$-inert. Every vertex in $(\xk) \setminus \omega_0$ has at most $\theta-k-1$ neighbors in $\omega_0$, so every vertex in $(\xk)\setminus \omega_0$ has at most $\theta-1$ neighbors in 
$$
\bigcup\{\{x\}\times K_n^2: \xi_{t - 1}(x)=0\}\cup \omega_0.
$$
Therefore, by the induction hypothesis, every vertex in  $(\xk)\setminus \omega_0$ has at most $\theta-1$ neighbors in $\omega_{t-1}$, so no vertex in $\xk$ becomes occupied at time $t$.
\end{proof}

\section{The subcritical regime for even threshold}

\newcommand{\step}{\rightarrowtail}

This section contains the proof of (\ref{subcritical-even}). Our 
argument is a suitable modification of the methods from 
\cite{GHS}, which are in turn based on 
duality-based construction of random 
surfaces \cite{DDGHS, ent, geom}. We cannot immediately apply 
the result from \cite{GM}, as we need to handle short-range 
dependence in the initial state. 

\subsection{Bootstrap percolation with obstacles}

Our focus will be the heterogeneous bootstrap percolation 
$\xi_t$, with a random initial set $\xi_0$. We will 
call such initial set a \df{positively correlated 
random field} if increasing events are positively correlated 
(that is, the FKG inequality holds), and \df{1-dependent}
if $\xi_0(x)$ and $\xi_0(y)$ are independent for 
$||x-y||_1\ge 2$.

\begin{theorem}\label{polluted}
Let $\p,\q >0$ be such that $\p+\q < 1$. Suppose $\xi_0$ has the following properties: for every $x\in \bZ^2$
\begin{equation}
\begin{aligned}
\prob{\xi_0(x) = 0} &= \p \\
\prob{\xi_0(x) = 2} &= 1- \p-\q\\
\prob{\xi_0(x) = 3} &= \q,
\end{aligned}
\end{equation}
and $\xi_0$ is a 1-dependent, positively correlated random field. 
Let $C>0$, and suppose that $\q>C \p^2$. Then for $C$ sufficiently large, we have that with probability at least $1-C\p^{3}$ either $\xi_\infty(\mathbf{0}) \ge 2$, or else $\mathbf{0}$ is contained in a cluster (maximal connected set) of sites $x\in \bZ^2$ with $\xi_\infty(x)=0$ that has 
$\ell^\infty$-diameter at most $1000$.
\end{theorem}

We will complete the proof of Theorem~\ref{polluted}, and then
the proof of (\ref{subcritical-even}), in Section 3.4. 
Throughout this section, we will assume that $\p$ is sufficiently small to make certain estimates work.

For a set $A\subset \bZ^2$, a configuration $\xi_0 \in \{0, \ldots, 5\}^{\bZ^2}$, and $k\in \{0, \ldots, 5\}$, define $\xi_0^{(A,k)}$ by
$$
\xi_0^{(A,k)}(x) = \begin{cases}
\xi_0(x) & \text{for $x\in A$}\\
k & \text{for $x\in A^c$}.
\end{cases}
$$
The resulting bootstrap dynamics, with initial configuration $\xi_0^{(A,k)}$, is denoted by $(\xi_t^{(A,k)})_{t\ge0}$. Observe 
that $(\xi_t^{(A,5)})_{t\ge0}$ is the heterogeneous bootstrap dynamics restricted to $A$, that is, run on the subgraph of
$\bZ^2$ induced by $A$. 
Also, for an $x\in \bZ^2$, let $\nonbrs(x, A)$ denote the number of neighbors of $x$ that lie in $A$.
\begin{prop} \label{super-sneaky}
Fix an integer $m\ge 1$. Fix a finite set $Z\subset \Z^2$ with $\nonbrs(x,Z^c)\le 2$ for every $x\in Z$, and
run two heterogeneous bootstrap percolation dynamics: the first with
initial configuration $\xi_0^{(Z,0)}$; the second with initial configuration $\xi_0^{(Z,5)}$.
Assume that the configuration $\xi_0$ on $Z$ satisfies the following
conditions.
\begin{enumerate}[label= \textup{(\roman*)}]
\item  Any $x\in Z$ with $\nonbrs(x, Z^c)= 2$ has $\xi_0(x) = 3$.
\item  For any $x\in Z$ with $\nonbrs(x,Z^c)\ge 1$, there is no
vertex $y$ with $\xi_0(y)=0$ within $\ell^\infty$-distance $m$ of $x$.
\item The final configuration in the dynamics started from the initial configuration $\xi_0^{(Z,5)}$ has no connected set
of vertices in state $0$ with $\ell^\infty$-diameter larger than $m/2$.
\end{enumerate}
Then, for all $t\ge 0$, we have
$$
\{x\in Z: \xi_t^{(Z,0)}(x)=0\} = \{x\in Z: \xi_t^{(Z,5)}(x)=0\}.
$$
\end{prop}
\begin{proof}
Assume the conclusion does not hold, and consider the
first time $t$ at which there exists a vertex $x\in Z$ such that $\xi_t^{(Z,0)}(x)=0$ but $\xi_t^{(Z,5)}(x)>0$. As the two dynamics have the same
initial configuration on $Z$, we have $t>0$.
By minimality of $t$, and properties (ii) and (iii), at time $t-1$ every $y\in Z$ such that $\nonbrs(y,Z^c)\ge 1$ has no neighbors in $Z$ with state $0$ in either dynamics. So, we cannot have $\nonbrs(x,Z^c) = 2$, since by (i), $\xi_{0}^{(Z,0)}(x) = 3$, and $x$ has at most two neighbors in state $0$ through time $t-1$, so the state of $x$ could not change at time $t$. We cannot have $\nonbrs(x,Z^c)=1$ either, since $\xi_{t-1}^{(Z,0)}(x)\ge 2$. Thus
$\nonbrs(x,Z^c)=0$, but then $x$ sees the same states among its neighbors in both dynamics at time $t-1$, and therefore $x$ has the same state in both dynamics at time $t$, a contradiction.
\end{proof}

\begin{lemma}\label{limited-bootstrap}
Fix an integer $s>0$, and let $N=\lfloor \p^{-s}\rfloor$.
Let $A=[-N,N]^2$.
With probability at least $1-C\p^{s}$, where $C=C(s)$ is a constant, 
 all connected clusters (maximal connected sets) of state $0$ vertices in $\xi_\infty^{(A,5)}$ have $\ell^\infty$-diameter at most $24s$.
\end{lemma}

\begin{proof}
First, replace all $3$s by $2$s in the initial configuration $\xi_0^{(A,5)}$; then, all connected clusters of $0$s in $\xi_\infty^{(A,5)}$ are rectangles. Fix an integer
$k>0$, and let $E_k$ be the
event that the
final configuration contains a rectangle of $0$s whose longest
side has length at least $k$. If $E_k$ occurs, $A$
contains an
internally spanned rectangle $R$ whose longest side length
is in the interval $[k/2,k]$ \cite{AL}. Then, any pair of neighboring lines, each
perpendicular to the longest
side of $R$, and such that both intersect $R$, must contain a state $0$ vertex within $R$ initially. Moreover, two pairs of neighboring lines that are at distance at least 2 from one another satisfy this requirement independently (since $\xi_0$ is 1-dependent). There are at most $(2N+1)^2 k^2$ possible
selections of the rectangle $R$. Therefore,
\begin{equation}\label{threshold2-1}
\P(E_k)\le 5\,N^2 k^2 (2k \p)^{k/6 - 1}\le \p^{(k-12s)/6 - 1} (2k)^{k/6+2},
\end{equation}
and the claim follows by choosing $k=24s$.
\end{proof}

Let $$L = \lfloor \delta/(m\p) \rfloor,$$
where $\delta>0$ is a small constant to be fixed later.
Also let $M = 12L$. Define the set
\begin{equation}\label{J-definition}
J = ([-m,m] \times [-M, M])\ \cup\ ([-M, M] \times [-m,m]).
\end{equation}
Call a vertex $x\in \Z^2$ \df{nice} if $\xi_0(x)=3$ and every vertex $y \in x + J$ has $\xi_0(y)\ge 2$. For each $u\in \Z^2$, define the rescaled box at $u$ to be
\[
Q_u := (2L+1) u + [-L,L]^2.
\]
We call a box $Q_u$ \df{good} if it contains a nice vertex. We will give a lower bound on the probability that a box is good. Call a vertex $x\in \bZ^2$ \df{viable} if every vertex $y \in x + J$ has $\xi_0(y)\ge 2$, and note that a viable vertex $x$ with $\xi_0(x) = 3$ is nice.

\begin{lemma}\label{viable-likely}
Fix a vertex $x\in \Z^2$ and an $\epsilon>0$. Assume
$\delta\le \epsilon/10^3$. Then,
\begin{equation}\label{viable-lb}
\P(x\ \text{\rm is viable})\ge 1-\epsilon.
\end{equation}
\end{lemma}

\begin{proof}
The argument is a simple estimate, where the first inequality below follows from the positive correlation assumption on $\xi_0$,
\begin{equation}\label{viable1}
\begin{aligned}
&\P\bigl(\xi_0(y) \ge 2 \text{ for all $y\in x+J$}\bigr)\\
&\ge [1-\p]^{2(2M+1)(2m+1)}\\
&\ge \exp{\bigl[-36\, mM\p\bigr]}\\&\ge\exp{(-500\,\delta)} ,
\end{aligned}
\end{equation}
provided $\p$ is small enough. Thus we can choose
any $\delta<\epsilon/500$ to make the probability in
(\ref{viable1}) larger than $1-\epsilon$.
\end{proof}

\begin{lemma}\label{good-likely} Fix any $\epsilon>0$,
and assume $\delta\le1/(4\cdot10^3)$. Then
there exists a constant $C=C(m,\epsilon,\delta)$, such that
$\q\ge C\p^2$ implies that
the probability that the box $Q_{\mathbf{0}}$ is good is at least
$1-\epsilon$.
\end{lemma}
\begin{proof}
For $k= 1, \ldots, \floor{\frac{2M+1}{3m}}-1$, let
$$
\mathtt{Row}_k = \bigl((-M + 3km) + [-m,m]\bigr) \times [-M,M]
$$
and
$$
\mathtt{Col}_k =  [-M,M] \times \bigl((-M + 3km) + [-m,m]\bigr).
$$
Define events
\begin{equation*}
\begin{aligned}
G_r &= \{\text{For at least } M/2m \text{ values of $k$, every $y\in \mathtt{Row}_k$ has $\xi_0(y)\ge2$}\}, \\
\text{and }\quad G_c &= \{\text{For at least } M/2m \text{ values of $k$, every $y\in \mathtt{Col}_k$ has $\xi_0(y)\ge2$}\}
\end{aligned}
\end{equation*}
The probability that $\mathtt{Row}_k$ has no $0$s is at least 
$3/4$, which can be proved by applying 
Lemma~\ref{viable-likely} with $\eps \le 1/4$. By large 
deviations for binomial random variables (noting that $
\mathtt{Row}_k$ and $\mathtt{Row}_{k+1}$ 
are at least distance 2 apart), we have
$$
\P(G_r) = \P(G_c) \ge 1 - \eps/4
$$
for small enough $\p$. By the assumed positive correlations in $\xi_0$, we have
\begin{equation*}
\P(G_r \cap G_c) \ge \P(G_r)\P(G_c) \ge 1- \eps /2,\\
\end{equation*}
and
\begin{equation*}
\begin{aligned}
\P(Q_{\mathbf{0}} \text{ is good}\ |\ G_r\cap G_c) &\ge \prob{\mathrm{Binomial}\left[ \left(\frac{M}{2m}\right)^2, \q \right] \ge 1 } \\
&\ge 1 - \exp(-\q(M/2m)^2)\\
&\ge 1 - \exp(-C\p^2 (3\delta/m^2\p)^2 )\\
& \ge 1-\eps/2
\end{aligned}
\end{equation*}
provided $C$ is large enough. The claim follows from the last two estimates.
\end{proof}

\subsection{Construction of a shell of good boxes}
 
 Let $B\subset\Z^2$. A site $u\in \bZ^2$ off the coordinate 
 axes is called \df{protected by $B$} provided that:
 \begin{itemize}
 \item if $u\in [1, \infty)^2\cup (-\infty, -1]^2$ then 
 both $ u+[-2,-1]\times [1,2]$ and 
 $u+[1,2]\times [-2,-1]$ intersect $B$; and 
 \item if 
 $u \in (-\infty, -1]\times [1,\infty)\cup [1,\infty)\times 
 (-\infty, -1]$, then both
 $ u+[-2,-1]\times [-2,-1]$ and 
 $u+[1,2]\times [1,2]$ intersect $B$.
 \end{itemize}
  If $u$ lies on one of the coordinate axes, we will not need to refer to $u$ as being protected.

A \df{shell $S$ of radius $r \in \N$} is defined to be a subset of $\Z^2$ that satisfies the following properties.
\begin{enumerate}[label=(S\arabic*)]
\item \label{S1} The shell $S$ contains all sites $u$ such that $\|u\|_1=r$ and $\|u\|_\infty \ge r-3$.
    (This implies that $S$ contains portions of the
    $\|\cdot\|_1$-sphere of radius $r$ in neighborhoods of
    each of the four sites $(\pm r, 0)$ and $(0,\pm r)$.)

\item \label{S2} For each $u\in S$, we have $r\le \|u\|_1 \le r+\sqrt{r}$
and
$\|u\|_\infty\le r$.

\item \label{S3} For each of the four directions $\varphi \in \{(\pm1, \pm1)\}$, there exists an integer $k = k(\varphi)\ge r/2$ such that $k\varphi \in S$.

\item \label{S4} If $u=(u_1,u_2)\in S$, and $|u_1| \ge 3$ and $|u_2|\ge 3$, then $u$ is protected by $S$.

\end{enumerate}

Let sites in the lattice $\Z^2$ be independently marked
black with probability $b$ and white otherwise. We wish to
consider paths of a certain type, and we start by defining
two types of steps. An ordered pair $u\step v$ of distinct sites $u,v\in \bZ^2$ is called:
\begin{enumerate}
\item a \df{taxed step} if each non-zero coordinate of $u$ increases in absolute value by $1$ to obtain the corresponding coordinate of $v$, while each zero coordinate of $u$ changes to $-1, 0$ or $1$ to obtain the corresponding coordinate of $v$;
\item a \df{free step} if  $\|v\|_1 < \|u\|_1$ and $v - u \in F$, where $F$ is the set of all vectors obtained by permuting coordinates and flipping signs from any of
\[
(1, 0), \text{ and } (2, 1).
\]
(For example, $(-1,2)\in F$.)
\end{enumerate}
Observe that, in a taxed step $u\step v$, we have 
$\|v\|_1>\|u\|_1$. We call $v-u$
the \df{direction} of either type of step. 

A \df{permissible path from $u_0$ to $u_k$} is a 
finite sequence of distinct sites $u_0, u_1, \ldots, u_k$ such that for every $i = 1, \ldots, k$, $u_{i-1}\step u_i$ is either a free step or a taxed step, and in the latter case, $u_i$ is white.  

To obtain a (random) shell $S$ of radius $r$, we let
\begin{equation}\label{eq:Adef}
A = \{v\in \Z^2 : \exists\, u \in \Z^2 \text{ with } \|u\|_1 < r \text{ and a permissible path from } u \text{ to } v\},
\end{equation}
and we define
\begin{equation}\label{eq:Sdef}
S = \{v \in \Z^2 \setminus A : \exists\, u \in A \text{ such that } u\step v \text{ is a taxed step}\}.
\end{equation}
Note that if $S$ exists, then all sites in $S$ must be black, since there are no permissible paths from $A$ to $A^c$.

\begin{prop} \label{shell}
Let $E_r$ be the event that there exists a shell of radius $r$ consisting of black sites.  There exists $b_1 \in (0,1)$ such that for any $b>b_1$ and $r\ge 1$, we have $\P(E_r)\ge 1/2$.
\end{prop}

Note that the event $E_r$ depends only on the colors of sites in $\{u\in \Z^2 : r\le \|u\|_1\le r+ \sqrt{r}\}$. However, in proving Proposition~\ref{shell}, we show that the set $S$ defined in~\eqref{eq:Sdef} is, in fact, the desired shell with large probability. The proof of the first lemma below, based on path counting, is nearly identical to the proofs of Lemmas 8, 9 and 10 in~\cite{GHS}, so we omit the details.

\begin{lemma}\label{shell-is-bounded}
There exists $b_2<1$ such that if $b>b_2$, then for each $r\ge 1$, the set $S$ defined by~\eqref{eq:Adef} and~\eqref{eq:Sdef} satisfies properties \ref{S1}, \ref{S2} and \ref{S3} with probability at least $1/2$.
\end{lemma}


\begin{lemma}\label{face-protected}
The set $S$ defined by~\eqref{eq:Adef} and~\eqref{eq:Sdef} satisfies property \ref{S4}.
\end{lemma}
\begin{proof}
Without loss of generality, suppose $u= (u_1,u_2)\in S$ is such that $u_i \ge 3$ for $i=1,2$, and by symmetry it suffices to show that $u+[1,2]\times [-2,-1]$ intersects $S$.  By the definition of $S$ in~\eqref{eq:Sdef}, $u$ must be reachable from $A$ by a taxed step. Since $u$ is not on a coordinate axis, the only site from which we can reach $u$ via a taxed step is $u + (-1,-1)$, so $u + (-1,-1)\in A$. Taking a free step in the direction $(1,-2)$ implies $u+(0,-3)\in A$ (this is where we require $|u_1|\ge 3$ and $|u_2|\ge 3$, to guarantee that direction $(1,-2)$ is, in fact, a free step). Observe that $u + (2,-1) \in A^c$, otherwise we would have $u\in A$, since it is reachable from this point by the free step in the direction $(-2,1)$.

Now their are two cases. If $u+(1,-2)\in A$, then $u+(2,-1)\in S$, since it is reachable from $u+(1,-2)$ along the taxed step 
in the direction $(1,1)$. Otherwise, if $u+(1,-2)\in A^c$, then $u+(1,-2)\in S$, since it is reachable from $u+(0,-3)\in A$ along the taxed step in the direction $(1,1)$. In either case, we have found a site in 
$(u+[1,2]\times [-2,-1])\cap S$.
\end{proof}

\begin{proof}[Proof of Proposition~\ref{shell}]
The claim follows from Lemmas~\ref{shell-is-bounded} and~\ref{face-protected}.
\end{proof}

\subsection{Construction of a protected set $Z$}
In this section we construct a set $Z\subset \bZ^2$, which is our candidate for the set satisfying the assumptions of Proposition~\ref{super-sneaky}.  

Suppose that there exists a shell $S$ of radius $r$
so that $Q_u$ is a good box for every $u\in S$. For every 
$u\in S$ with both coordinates at least $3$ in absolute value, select a 
nice vertex from $Q_u$ and gather the selected vertices 
into the set $U$. (No nice vertices are chosen from $Q_u$ if at least one coordinate of $u\in S$ is less than 3 in absolute value.)

A \df{fortress} is a square of side length
$12L+1$ (this is the reason for our choice of $M=12L$ in the definition of $J$ at~\eqref{J-definition}), all four of whose corners are nice. Suppose that
there is a fortress centered at each of the four vertices $(\pm r(2L+1),0), (0,\pm r(2L+1))$.
Let $K$ be the set of all corner vertices of all fortresses ($16$ in all).
For $x\in \bZ^2$, define $\mathtt{Rect}(x)$ to be the rectangle with opposite corners at $x$ and $\mathbf{0}$ (for example, if $x=(x_1,x_2)$ with $x_1\ge0$ and $x_2\leq 0$, then $\mathtt{Rect}(x) = [0,x_1]\times[x_2,0]$). Now define $Z$ by
\begin{equation}\label{eqn:Zdef}
Z=\bigcup_{x \in U\cup K} \mathtt{Rect}(x).
\end{equation}
Note that by construction, all convex corners of $Z$ are nice vertices, and near each of the coordinate axes, there are two nice vertices on the line orthogonal to the nearby axis that are at distance $12L+1$. In addition, the fact that the slope of $S$ is locally bounded above and below (by property~\ref{S4}) makes the following proposition geometrically transparent. The formal proof is very similar to the proofs of Lemmas 20 through 26 in~\cite{GHS}, though it is much simpler, and is omitted.

\begin{lemma}\label{Z-is-protected}
Suppose $Z$ is defined as in~\eqref{eqn:Zdef}. If $\p$ is sufficiently small (depending on $\delta$ and $m$) to make $L$ sufficiently large, then $Z$ satisfies assumptions (i) and (ii) of Proposition~\ref{super-sneaky}.
\end{lemma}

\subsection{Existence of a protected set $Z$}

Assume $N_0=3\lfloor \p^{-36}\rfloor$, 
$n_0=\lfloor  \p^{-19}\rfloor$, 
$T=\lfloor  \p^{-17}\rfloor$, and 
$\Delta=\lfloor \p^{-19}\rfloor$. 
Define the sequence of separated annuli 
$$A_i=\{x\in \bZ^2: n_0+(2i-1)\Delta\le \|x\|_1\le 
n_0+2i\Delta\}, 
$$
for $i=1,\ldots, T$. 

\begin{lemma} 
\label{Z-exists}
Fix an $m$. 
For a small enough $\epsilon>0$ and $\delta>0$, 
and $\q\ge C\p^2$, where $C$ is given 
in Lemma~\ref{good-likely}, the following holds. 
With probability at least $1-\exp(-1/(4\p))$,  
there exists a protected set $Z$ satisfying assumptions (i) and (ii) of Proposition~\ref{super-sneaky} contained in $\{x\in \bZ^2: \|x\|_1\le N_0\}$. 
\end{lemma}

\begin{proof}
Note that $n_0+2T\Delta\le N_0$. 

Paint each site $x\in \bZ^2$ \emph{black}
if the box $Q_x$ is good. Let $r_i = \floor{(n_0+(2i-1)\Delta)/ (2L+1)}+11$, so $(2L+1)r_i -20L \ge n_0+(2i-1)\Delta$, and observe that $\sqrt{r_i} \le \sqrt{N_0 / L} \ll \Delta / (2L+1)$ for $\p$ small. Therefore, existence of a shell of good boxes of radius $r_i$ depends only on the states of vertices within the annulus $A_i$.  Moreover, we have that sites $x_1$ and $x_2$ with 
$\|x_1-x_2\|_\infty\ge 30$ are painted independently, and so 
by \cite{LSS} the configuration of black sites 
dominates a product measure of density $b_1$ (chosen 
from Proposition~\ref{shell}) provided $\epsilon>0$ 
in Lemma~\ref{good-likely} is small enough, and $\delta$ 
is chosen appropriately. It follows that, when $\p$ is small enough, by Proposition~\ref{shell}, a shell  of good boxes of radius $r_i$ exists with probability at least $1/2$. The existence of a shell of good boxes of radius $r_i$ is an increasing event (in $\xi_0$), and so it is positively correlated with existence of nice vertices at the 16 locations comprising the set $K$ ($\subset A_i$) in~\eqref{eqn:Zdef}.
Therefore, the set $Z$ given by~\eqref{eqn:Zdef} exists with convex corners $U\cup K\subset A_i$ with probability 
at least $\p^{16}/2$. Due to the separation of shells, 
the probability that such a $Z$ does not exist in $A_i$ for all $i=1,\ldots, T$
is then at most $(1-\p^{16}/2)^{\p^{-17}/2}\le \exp(-1/(4\p))$. By Lemma~\ref{Z-is-protected}, if $Z$ constructed in this manner exists, then it satisfies assumptions (i) and (ii) of Proposition~\ref{super-sneaky}.
\end{proof}

\begin{proof}[Proof of Theorem~\ref{polluted}]
Choose $s=37$ in Lemma~\ref{limited-bootstrap}. 
That determines $m=48s < 2000 $. The proof is concluded by Lemma~\ref{Z-exists}, Lemma~\ref{limited-bootstrap}, 
 and
Proposition~\ref{super-sneaky}. 
\end{proof}

\begin{proof}[Proof of Theorem~\ref{sharp transition even} equation (\ref{subcritical-even})]
Initialize $\xi_0$ using inertness as in 
Lemma~\ref{omega upper bound}, then convert all $1$s to $0$s,
and all $4$s and $5$s to $3s$. Suppose $v_0 \in \mathbf{0}\times K_n^2$. If $v_0\in \omega_\infty$, then either $v_0\in \omega_0$, or some Hamming square in  $\{\{x\} \times K_n^2 : x\in [-1000,1000]^2\}$ is not $\theta$-inert, or else $\mathbf{0}$ is in a cluster of state-$0$ sites in $\xi_\infty$ that has diameter larger than $1000$.
Therefore, by Theorem~\ref{polluted} and Lemma~\ref{inertness-even}
\begin{equation*}
\begin{aligned}
&\probsub{p}{v_0\in \omega_\infty}\\
&\le 
\probsub{p}{v_0\in \omega_0}+10^7\probsub{p}{\mathbf{0}\times K_n^2
\text{ is not $\theta$-inert}}+
C\probsub{p}{\mathbf{0}\times K_n^2
\text{ is not $(\theta-1)$-inert}}^3\\
&=n^{-2/\ell+o(1)}.
\end{aligned}
\end{equation*}
The lower bound is easy: by Lemma~\ref{internal-results-even} part 5,
$$
\probsub{p}{v_0\in \omega_\infty}
\ge \probsub{p}{\mathbf{0}\times K_n^2
\text{ is $\theta$-IS}} = n^{-2/\ell + o(1)},
$$
and (\ref{subcritical-even}) is thus proved.
\end{proof}

\section{The supercritical regime for even threshold}

\renewcommand{\frame}{\text{\tt Frame}}
\newcommand{\blocking}{\text{\tt Blocking\_In }}

In this section, we prove the claims of 
Theorem~\ref{sharp transition even} 
when $a^\ell\ge 2(\ell-1)!$. In the following subsections, we prove, 
in order: (\ref{supercritical-even}), upper bound on the rate 
(\ref{supercritical-even1}) for $\ell\ge 2$, lower bound 
on the same rate for $\ell\ge 2$, and the asymptotics for
the exceptional case $\ell=1$. 

\subsection{Comparison process and rescaling}

Initialize the comparison process, $\xi_t$, as follows. For $x\in \bZ^2$, let
\begin{equation}
\xi_0(x) = \begin{cases}
0 & \text{if $\{x\}\times K_n^2$ is $\theta$-IS} \\
k & \text{if $k\in \{1,2\}$ and $\{x\}\times K_n^2$ is $(\theta-k)$-IS, but is not $(\theta-k+1)$-IS}\\
5 & \text{if $\{x\}\times K_n^2$ is not $(\theta-2)$-IS.}
\end{cases}
\end{equation}
In other words, initialize $\xi_t$ as in Lemma~\ref{omega lower bound}, but replace all $3$s and $4$s with $5$s. 

To apply Lemma~\ref{omega lower bound}, we need a method to 
show that 
$\probsub{p}{\xi_\infty(\mathbf{0})=0}$ is close to 1, and 
for that, we adapt the rescaling from \cite{GM} to our purposes;
in particular, we need to account for 
the existence of $1$s, which require activation from $0$s, 
and to
prove high final density at the critical value (when $a^\ell = 2 (\ell-1)!$). We let 
\begin{equation}\label{N-def}
N=
\begin{cases}
\floor{n^{1/\ell} (\log n)^{-1/2\ell}} & \ell\ge 2\\
\floor{n (\log n)^{-3/4}} & \ell=1
\end{cases}
\end{equation}
and, for $x \in \bZ^2$, let $\Lambda_x = N\cdot x + [0,N-1]^2$ be the $N\times N$ box in $\bZ^2$ with lower-left corner at $N x$. 
Call the box $\Lambda_x$ {\em good} if $\xi_0(y) \le 2$ for every $y\in \Lambda_x$ and, in addition, every row and column of $\Lambda_x$ contains at least one $y$ such that $\xi_0(y) \le 1$.  Call a box {\em very good} if $\xi_0(y)\le 1$ for every $y\in \Lambda_x$ and $\xi_0(y) = 0$ for some $y\in \Lambda_x$. 

\begin{lemma} \label{not-good-bound}
For $\ell\ge 1$ and large enough $n$, 
$$
\probsub{p}{\Lambda_x \text{ is not good}} \leq
6 n^{(2/\ell) - (a^\ell/\ell!)} \cdot(\log n)^{- (1/\ell\wedge 1/2)}.
$$
\end{lemma}

\begin{proof}
By Lemma~\ref{internal-results-even}, for $\ell\ge 2$,
\begin{equation*}
\begin{aligned}
\probsub{p}{\Lambda_x \text{ is not good}} &\leq N^2 \probsub{p}{\xi_0(\mathbf{0}) = 5} + 2N\cdot \left(1- \probsub{p}{\xi_0(\mathbf{0}) \le 1}\right)^N \\
&\leq 3N^2 n^{-a^{\ell}/{\ell}!} + 2N \exp\left[-N \cdot \frac {2a^{\ell+1}}{(\ell+1)!}\cdot \frac{(\log n)^{1+1/\ell}}{n^{1/\ell}}(1+o(1))\right] \\
&\leq 3 n^{(2/\ell) - (a^\ell/\ell!)} \cdot(\log n)^{- 1/\ell} + n^{1/\ell} \exp\left[-C (\log n)^{1 + 1/2\ell}\right].
\end{aligned}
\end{equation*}
When $\ell=1$, repeat the above computation with $\probsub{p}{\xi_0(\mathbf{0}) = 5} \le 3a n^{-a} \log n$.
\end{proof}

\begin{proof}[Proof of (\ref{supercritical-even})]
It follows from Lemma~\ref{omega lower bound} that 
$$
\bigcup\{\{x\}\times K_n^2 : \xi_\infty(x)=0\}\subset\omega_\infty,
$$
so we need only to show that $\probsub{p}{\xi_\infty(\mathbf{0})=0} \to 1$ when $a^\ell \ge 2(\ell-1)!$. Let $\cC_0$ denote the cluster of good boxes containing the box $\Lambda_0$. Observe that
$$
\probsub{p}{\abs{\cC_0} = \infty} = \probsub{p}{\{\abs{\cC_0} = \infty\}\cap \{\text{$\cC_0$ contains a very good box}\}} \le \probsub{p}{\xi_\infty(\mathbf{0}) = 0}.
$$
The last inequality follows from the fact that a very good box in $\cC_0$ sets off a cascade resulting in all vertices in $\cC_0$ eventually flipping to $0$. Now, Lemma~\ref{not-good-bound} implies $\probsub{p}{\abs{\cC_0} = \infty} \to 1$. 
\end{proof}

\subsection{Upper bound in (\ref{supercritical-even1}) for $\ell\ge 2$}

Throughout this subsection, assume that $\ell\ge 2$, 
$a^\ell/\ell!> 2/\ell$
and that $\xi_0$ is built by internal spanning 
properties, as in Lemma~\ref{omega lower bound}.

We will prove first the upper bound on the rate.
\begin{lemma} \label{even-supper-rate-upper}
The probability that the Hamming square based 
at the origin is not completely filled satisfies the 
following bound:
\begin{equation}\label{even-super-rate-less}
\probsub{p}{\{\mathbf{0}\}\times K_n^2 \not\subset \omega_\infty}
\le n^{4/\ell-4a^\ell/\ell!+o(1)}.
\end{equation}
\end{lemma}

For a deterministic or random set $A\subset \bZ^2$, we say that 
the event $\blocking A$ happens if 
there exists a rectangle $R=[a_1,a_2]\times [b_1,b_2]$ so that:
$\mathbf{0}\in R$; $R$ is nondegenerate, i.e., $a_1<a_2$ and $b_1<b_2$; 
and each of the four sides of $R$, $\{a_1\}\times[b_1,b_2]$, 
$\{a_2\}\times[b_1,b_2]$,
$[a_1,a_2]\times\{b_1\}$, and $[a_1,a_2]\times\{b_2\}$,
either contains two distinct sites in $A$ with $\xi_0$-state $3$ 
or a site in $A$ with $\xi_0$-state $4$. A \emph{frame} is such 
a rectangle $R$ whose four corners have $\xi_0$-state $3$. 

\begin{lemma}\label{frame-necessary} Suppose that 
$\xi_\infty(\mathbf{0})\ne 0$. Assume that there is  circuit of $0$s around $\mathbf{0}$ in $\xi_t$, for some $t$. 
Let $A$ comprise sites in the strict interior of this circuit.
Assume that there are no sites in $A$ with $\xi_0$-state $5$, and there is 
at most one site in $A$ with $\xi_0$-state $4$. 
Then the event $\blocking A$ happens. 
\end{lemma} 

\begin{proof} We may assume that all sites in $A^c$ are $0$s
in $\xi_0$. Let $A'$ be the set of sites which are non-zero in 
$\xi_\infty$. Then the leftmost and the rightmost 
site on the top line of $A'$ must either be the same state with $\xi_0$-state 
$4$, or be two distinct sites which both have $\xi_0$-state at least $3$. 
To check nondegeneracy, assume that, say, $b_1=b_2$. 
As there are no sites in $\xi_0$-state $5$ in $A$, 
there then must be two sites at $\xi_0$-state $4$ on either side of
$\mathbf{0}$ on the $x$-axis, but by the assumption there can be at most one such site. 
\end{proof}

Now we pick $N$ as in (\ref{N-def}) and also 
keep the definition of good boxes from the previous
subsection. 
For a constant $D$, let $G_1(D)$ be the event that 
there is a circuit of good boxes that encircles $\mathbf{0}$, is contained 
in $[-DN,DN]^2$, and is connected to the infinite cluster of good boxes. 

\begin{lemma}\label{circuit-good-sites}
For any $L$ there is a constant $D=D(a,L)$ so that 
\begin{equation}\label{circuit-good-sites-eq}
\probsub{p}{G_1(D)^c}\le n^{-L}
\end{equation}
\end{lemma}
\begin{proof}
This follows from Lemma~\ref{not-good-bound}, together 
with a standard percolation argument. 
\end{proof}

\begin{lemma}\label{no-4-5}
The probability that  $[-DN,DN]^2$ contains at least one site 
with $\xi_0$-state $5$ 
or at least two sites in $A$ with $\xi_0$-state $4$
is $n^{4/\ell-4a^\ell/\ell!+o(1)}$
\end{lemma}
\begin{proof}
This follows from Lemma~\ref{internal-results-odd}. 
\end{proof}

 

\begin{lemma}\label{frame-bound}
Assume $D$ is a fixed constant. Then $$\probsub{p}{\blocking [-DN,DN]^2}\le  n^{4/\ell-4a^\ell/\ell!+o(1)}.$$
\end{lemma}

\begin{proof}
Define $\lambda$ so that $DN=n^\lambda$, so that $\lambda=1/\ell+o(1)$, and 
let $\alpha=a^\ell/\ell!$. Note that $2\lambda<\alpha$. We will restrict all our sites to 
the region $[-DN,DN]^2$. 
Let $\frame$ be the event that a frame exists (which thus by definition 
means existence in $[-DN,DN]^2$).
Then $\probsub{p}{\frame)=\Theta(n^{4\lambda-4\alpha}}$.

The event that there exists a nondegenerate rectangle $R$ that has at least 
two sites with $\xi_0$-state $3$ on all sides can be 
split into the following events, according to additional properties
of the configuration on $R$:
\begin{itemize}
\item $R$ is a frame; 
\item $R$ has no $3$s at the corners (i.e., there is 
no sharing), which 
happens with probability at most a constant times
$$n^{4\lambda}(n^{2\lambda}n^{-2\alpha})^4=n^{12\lambda-8\alpha}=o(\probsub{p}{\frame})$$
(we give these probabilities as products, 
reflecting successive choices:  
four lines determining $R$, pairs of points on the same line away from 
corners; single points on lines away from corners, states at corners); 
\item $R$ has exactly one $3$ at a corner, 
with probability 
at most a constant times
$$n^{4\lambda}(n^{2\lambda}n^{-2\alpha})^2(n^{\lambda}n^{-\alpha})^2n^{-\alpha}
=n^{10\lambda}n^{-7\alpha}=
o(\probsub{p}{\frame});$$ 
\item $R$ has exactly two corner $3$s on the same line, with probability at most a constant times
$$n^{4\lambda}(n^{2\lambda}n^{-2\alpha})(n^{\lambda}n^{-\alpha})^2n^{-2\alpha}
=n^{8\lambda}n^{-6\alpha}=
o(\probsub{p}{\frame});$$
\item $R$ has exactly two corner $3$s not on the same line, with probability at most a constant times 
$$n^{4\lambda}(n^{\lambda}n^{-\alpha})^4n^{-2\alpha}
=n^{8\lambda}n^{-6\alpha}=
o(\probsub{p}{\frame});$$
\item $R$ has exactly three corner $3$s, with probability
 at most a constant times
$$n^{4\lambda}(n^{\lambda}n^{-\alpha})^2n^{-3\alpha}
=n^{6\lambda}n^{-5\alpha}=
o(\probsub{p}{\frame}).$$
\end{itemize}
Next we consider the event that a rectangle $R$ has 
exactly one $4$ on its boundary, and either two $3$s or a 
$4$ on each of its sides. Again, we split this event according 
to additional properties:
\begin{itemize}
\item
 $4$ is not at a corner of $R$ and neither are $3$s, with  probability at most a constant times
$$n^{4\lambda}(n^{2\lambda}n^{-2\alpha})^3n^{\lambda}n^{-2\alpha}
=n^{11\lambda}n^{-8\alpha}=
o(\probsub{p}{\frame});$$
\item the $4$ is at a corner of $R$, but no $3$s are at corners, 
with  probability at most a constant times 
$$n^{4\lambda}(n^{2\lambda}n^{-2\alpha})^2n^{-2\alpha}
=n^{8\lambda}n^{-6\alpha}=
o(\P(\frame));$$
\item the $4$ is at a corner of $R$, and a $3$ is at the opposite 
corner, with probability at most a constant times
$$n^{4\lambda}(n^{\lambda}n^{-\alpha})^2n^{-2\alpha}n^{-\alpha}
=n^{6\lambda}n^{-5\alpha}=
o(\probsub{p}{\frame}).$$
\end{itemize}
Together with Lemma~\ref{no-4-5}, these calculations end the 
proof.
\end{proof}
\begin{proof}[Proof of Lemma~\ref{even-supper-rate-upper}]
Choose the constant $D$ in Lemma~\ref{circuit-good-sites}
so that $L$ in (\ref{circuit-good-sites-eq}) satisfies
$L>4a^\ell/\ell!-4/\ell$. Then (\ref{even-super-rate-less}) 
follows from Lemmas~\ref{frame-necessary}--\ref{frame-bound}.
\end{proof}

\subsection{Lower bound in (\ref{supercritical-even1}) for $\ell\ge 2$}
In this subsection also, we 
assume that $a^\ell/\ell!>2/\ell$
but now $\xi_0$ is built by inertness
properties, as in Lemma~\ref{omega upper bound}. 
In this section, we prove the lower bound on the rate. 

\begin{lemma} \label{even-supper-rate-lower}
The probability that the configuration on the Hamming square based 
at the origin never changes satisfies the 
following bound:
\begin{equation}\label{even-super-rate-more}
\probsub{p}{\omega_\infty=\omega_0\text{ on }\{\mathbf{0}\}\times K_n^2}
\ge n^{4/\ell-4a^\ell/\ell!+o(1)}.
\end{equation}
\end{lemma}

Fix a non-degenerate rectangle $R$. Let $\xi_0^0$ be obtained from 
$\xi_0$ by converting all $4$s and $5$s to $3$s on $R$, and changing all sites 
to $0$ off $R$. 
Let $\xi_t^0$ 
be the bootstrap dynamics 
started from this initial state. 
We say that $R$ is \emph{protected}
if $R$ has its four corners in 
$\xi_0^0$-state 
$3$, no site in $R$ has $\xi_0^0$-state $0$ 
and no site on the boundary of $R$ has $\xi_0^0$-state $1$. 

\begin{lemma} \label{super-sneaky-v2} 
Assume a nondegenerate rectangle $R$ is protected. 
Then no site ever changes state in $\xi_t^0$, and 
therefore $\xi_t$ never changes any state in $R$.
\end{lemma}
\begin{proof}
The first site to change state would have to be on 
the boundary of $R$, which is clearly impossible. 
\end{proof}

Assume now $N=\lfloor n^{1/\ell}/\log^5n\rfloor$. 
Define the following two events:
\begin{equation*}
\begin{aligned}
&G_1=\{\text{there exists a rectangle $R$ with }
\mathbf{0}\in R\subset[-N,N]^2,  \text{four corners in $\xi_0^0$-state 3, }\\
& \qquad\qquad \text{and no site on the boundary of $R$ is in $\xi_0^0$-state 0 or 1}\},\\
&G_2=\{\text{there is no $x\in [-N,N]^2$ with  $\xi_0(x)=0$}\}.
\end{aligned}
\end{equation*}
\begin{lemma} \label{protected-lower} 
With our choice of $N$, 
$$\probsub{p}{G_1}\ge n^{4/\ell-4a^\ell/\ell!+o(1)}.
$$
\end{lemma}
\begin{proof}
This follows from an argument that is very similar to the
one for Lemma~\ref{good-likely}.
\end{proof}

\begin{lemma} \label{no-0} 
With our choice of $N$,
$$
\probsub{p}{G_2^c}\to 0,
$$ 
as $n\to\infty$.
\end{lemma}
\begin{proof}
This follows from 
Lemma~\ref{internal-results-even} and 
Lemma~\ref{inertness-even}.
\end{proof}

\begin{proof}[Proof of Lemma~\ref{even-supper-rate-lower}]
Observe that $G_1$ and $G_2$ are increasing events, therefore
by FKG and  Lemmas~\ref{protected-lower} and \ref{no-0}, 
$$
\probsub{p}{G_1\cap G_2}\ge n^{4/\ell-4a^\ell/\ell!+o(1)},
$$
and the result follows from 
Lemma~\ref{super-sneaky-v2}.
\end{proof}

\subsection{The exceptional case: $\theta=4$}
\newcommand{\tblocking}{\text{\tt 4\_Blocking\_In }}

We assume that $\theta=4$ throughout this 
section, and that, in accordance with (\ref{form-p}), 
$$
p=a\frac{\log n}{n^2},
$$
with $a>2$. We first prove an analogue of Lemma~\ref{even-supper-rate-upper}. We will again assume that $\xi_0$ is built by internal spanning 
properties, as in Lemma~\ref{omega lower bound}, and observe that
the sites 
with $\xi_0$-state $4$ and $\xi_0$-state $3$, 
both of which we call \emph{4-obstacles}, are comparably 
improbable at our precision level. (Also note that 
there are no sites with $\xi_0$-state 5.) As a result, the 
convergence rate changes.

\begin{lemma} \label{theta-4-rate-upper}
The probability that the Hamming square based 
at the origin is not completely filled satisfies the 
following bound:
\begin{equation}\label{theta-4-rate-less}
\probsub{p}{\{\mathbf{0}\}\times K_n^2 \not\subset \omega_\infty}
\le n^{-2a+o(1)}.
\end{equation}
\end{lemma}

For a set $A\subset \bZ^2$, we say that 
the event $\tblocking A$ happens if 
there exists a rectangle $R=[a_1,a_2]\times [b_1,b_2]$ so that
$\mathbf{0}\in R$ and either:
\begin{itemize}
\item $a_2-a_1\ge 3$ and $b_2-b_1\ge 3$ and two layers on the 
four sides of $R$, $[a_1,a_1+1]\times[b_1,b_2]$, 
$[a_2-1,a_2]\times[b_1,b_2]$,
$[a_1,a_2]\times[b_1,b_1+1]$, and $[a_1,a_2]\times[b_2-1,b_2]$,
each contain at least two 4-obstacles in $A$;
\item $0\le a_2-a_1\le 2$, $b_2-b_1\ge 3$, and 
$R$ contains 4 or more 4-obstacles in $A$;
\item $a_2-a_1\ge 3$, $0\le b_2-b_1\le 2$, and 
$R$ contains 4 or more 4-obstacles in $A$; or
\item 
$0\le a_2-a_1\le 2$,  $0\le b_2-b_1\le 2$,
and 
$R$ contains 2 or more 4-obstacles in $A$. 
\end{itemize}

\begin{lemma}\label{4-blocking-necessary} 
Suppose that 
$\xi_\infty(\mathbf{0})\ne 0$. Assume that there is  circuit of $0$s around $\mathbf{0}$ in $\xi_t$, for some $t$. 
Let $A$ comprise sites in the strict interior of this circuit.
Then the event $\tblocking A$ happens. 
\end{lemma} 

\begin{proof} As before, we may assume that all sites 
in $A^c$ are $0$s
in $\xi_0$ and let $A'$ be the set of sites which are non-zero in 
$\xi_\infty$. If the top line of $A'$ 
consists of a single 4-obstacle, 
then the next line from the top must also contain a 4-obstacle. (Otherwise, the next line from the top would eventually turn into all $0s$, causing the solitary $4$-obstacle on the top line to be surrounded by $0$s.)
Finally, if there is 
a single 4-obstacle within $R$, then 
all sites in $R$ eventually turn into $0$s.
\end{proof}

We next note that Lemma~\ref{circuit-good-sites} still
holds, with $N$ given by~\eqref{N-def} with $\ell=1$, and proceed
with our final lemma.

\begin{lemma}\label{4-blocking-bound}
Assume $D$ is a fixed constant. Then $\probsub{p}{\tblocking [-DN,DN]^2}\le  n^{-2a+o(1)}$.
\end{lemma}

\begin{proof}
For the event $\{\tblocking [-DN,DN]^2\}$ to happen, one of the four 
events, corresponding to the four items in its 
definition, must happen.
The event in the 
first item happens with probability at most $n^{4-4a+o(1)}$, 
as in the proof of Lemma~\ref{frame-bound}. The events 
in the second and 
third item also happen with probability at most 
$n^{4-4a+o(1)}$. The event in the last item happens with 
probability $n^{-2a+o(1)}$, and this last probability is 
the largest, as $a>2$. 
\end{proof}

\begin{proof}[Proof of Lemma~\ref{theta-4-rate-upper}]
Analogously to the case of even $\theta\ge 6$, choose the constant 
$D$ in Lemma~\ref{circuit-good-sites}
so that $L$ in (\ref{circuit-good-sites-eq}) satisfies
$L>2a$, and use Lemmas~\ref{4-blocking-necessary} and 
\ref{4-blocking-bound} to conclude (\ref{theta-4-rate-less}).
\end{proof}

We conclude this section by the simple observation that gives
the matching lower bound. 

\begin{lemma} \label{theta-4-rate-lower}
The Hamming square based 
at the origin remains unoccupied forever with 
probability bounded below as follows:
\begin{equation}\label{theta-4-rate-more}
\probsub{p}{\omega_\infty\equiv 0\text{ on }\{\mathbf{0}\}\times K_n^2}
\ge n^{-2a}(1+o(1)).
\end{equation}
\end{lemma}

\begin{proof}
The inclusion
$$
\{\omega_0\equiv 0\text{ on }\{\mathbf{0}, (0,1)\}\times K_n^2\}
\subset
\{\omega_\infty\equiv 0\text{ on }\{\mathbf{0}\}\times K_n^2\}.
$$
gives the desired bound.
\end{proof}

\section{The odd threshold}
\newcommand{\pn}{p^{(n)}}
\newcommand{\bxi}{\overline\xi}
\newcommand{\hxi}{\widehat\xi}
\newcommand{\greenp}{\text{\tt Green\_Percolation}}
\newcommand{\zerop}{\text{\tt Green\_Connection}}
\newcommand{\nonred}{\text{\tt Nonred\_Percolation}}
\newcommand{\nonredzero}{\text{\tt Nonred\_Connection}}
\newcommand{\redcircuit}{\text{\tt Red\_Circuit}}
\newcommand{\nozero}{\text{\tt No\_Zero}}
\newcommand{\obstbox}{\text{\tt Obstacle\_Box}}

In this section we prove Theorem~\ref{odd-theta-thm}. 
In the first three subsections, we handle the case 
$\ell\ge 2$: first we define, and give bounds for, 
the critical value $a_c$, then we prove (\ref{odd-theta-super}), 
and then (\ref{odd-theta-sub}). In the last, fourth subsection, 
we sketch the argument for the case $\ell=1$ in lesser detail. 

\subsection{The critical value of $a$ for $\ell\ge 2$}

Pick an $a>0$ and an $\epsilon\in(0,\exp\left[-\frac{ a^\ell}{\ell!}\right]-\exp\left[-\frac{2 a^\ell}{\ell!}\right])$.
Consider the initial state $\xi^{(a,\epsilon)}_0$ given by the 
product measure with 
\begin{equation*}
\begin{aligned}
\P(\xi^{(a,\epsilon)}_0(x)=0)&=\epsilon,\\
\P(\xi^{(a,\epsilon)}_0(x)=1)&=\left(1 - e^{-a^{\ell}\slash{\ell!}}\right)^2, \\
\P(\xi^{(a,\epsilon)}_0(x)=3)&=\exp\left[-\frac{2 a^\ell}{\ell!}\right],\\
\P(\xi^{(a,\epsilon)}_0(x)=2)&=1-P(\xi^{(a,\epsilon)}_0(x)=0)-P(\xi^{(a,\epsilon)}_0(x)=1)-P(\xi^{(a,\epsilon)}_0(x)=3)
\end{aligned} 
\end{equation*}
for every $x\in\bZ^2$. We will call this an \emph{$(a,\epsilon)$-initialization} and denote the resulting bootstrap dynamics 
by $\xi_t^{(a,\epsilon)}$. 

Define $a_c\in [0,\infty]$ as follows: 
$$
a_c=\inf\{a> 0: \lim_{\epsilon\to 0}\P(\xi^{(a,\epsilon)}_\infty(\mathbf{0})=0)>0\}.
$$
Observe that $\P(\xi_\infty^{(a,\epsilon)}(\mathbf{0})=0)$ is a
nonincreasing function of $\epsilon$ and therefore its limit  
as $\epsilon\to 0$ exists. Furthermore, this limit is a nondecreasing function of $a$, 
and therefore it vanishes on $[0,\alpha_c)$ and is strictly positive on $(a_c,\infty)$. 

The next two lemmas establish that $a_c$ is nontrivial, 
that is, $a_c\in (0,\infty)$, by comparison to
the critical value $p_c^{\text{site}}$ of site percolation on $\bZ^2$, and to the critical value of the site percolation 
on the triangular lattice.

\begin{lemma} \label{ac-pc} The following strict inequality holds:
\begin{equation}\label{ac-pc-eq}
(1 - e^{-a_c^{\ell}\slash{\ell!}})^2<p_c^{\text{\rm site}}.
\end{equation}
 In particular, $a_c<\infty$. Furthermore, $\lim_{\epsilon\to 0}\P(\xi^{(a,\epsilon)}_\infty(\mathbf{0})=0)\to 1$ as $a\to\infty$.
\end{lemma}
\begin{proof} 
Given a configuration $\xi_0=\xi_0^{(a,\epsilon)}$, form the following set of \emph{green} sites. 
Any site $x$ with $\xi_0(x)\le 1$ is green. Also make green any site $x$ such that 
$\xi_0(x)=2$ and $\xi_0(y)\le1$ for
all sites $y$ among the $8$ nearest neighbors of $x$, except possibly for two diagonally opposite 
neighbors. That is, if the local configuration in $\xi_0$ around a site $x$ is
\begin{equation}\label{ac-pc-eq1}
\begin{matrix}
1 \,1\, *\\
1\, 2 \,1\\
* \,1 \,1\\
\end{matrix} \qquad
\text{ or } \qquad
\begin{matrix}
* \,1\, 1\\
1\, 2 \,1\\
1 \,1 \,*\\
\end{matrix} \quad,
\end{equation}
where $*$ denotes an arbitrary state, then $x$ is green, and it is also green 
if its local configuration has $0$s in place of any of the $1$s
in (\ref{ac-pc-eq1}). Let $\greenp$ 
be the event that $\mathbf{0}$ is in an infinite connected set of green sites, 
and $\zerop$ the event that $\mathbf{0}$ is green and connected to a vertex with state $0$ in $\xi_0$ through green sites. 
Then 
\begin{equation}\label{ac-pc-eq2}
\P(\greenp\setminus\zerop)=0.
\end{equation}
Moreover, we claim that 
\begin{equation}\label{ac-pc-eq3}
\zerop\subset \{\xi_\infty(\mathbf{0})=0\}.
\end{equation}
To see this, consider the set of all sites in a connected cluster $\cC$ of 
$\mathbf{0}$ of green sites that includes a $0$ in $\xi_0$. Let $\cC_0$ be the 
set of all sites in $\cC$ that eventually assume state $0$. If $\cC_0\subsetneqq\cC$, then 
there exist neighbors $x$ and $y$ with $x\in \cC_0$ and 
$y\in \cC\setminus\cC_0$. But then $\xi_0(y)=2$, 
and by inspection of the configurations in (\ref{ac-pc-eq1}), we see that $y$ must have at least $2$ 
neighbors in $\cC_0$, a contradiction. Therefore 
 $\cC_0=\cC$ and 
(\ref{ac-pc-eq3}) holds. 

Finally, it follows from \cite{AG} (see also \cite{BBR}) that there exists 
an $a$ with $(1 - e^{-a^{\ell}\slash{\ell!}})^2<p_c^{\text{site}}$, so that 
$\P(\greenp)>0$. This, together with (\ref{ac-pc-eq1}--\ref{ac-pc-eq3}), 
establishes (\ref{ac-pc-eq}). Moreover, it follows from 
standard percolation arguments that 
$\P(\greenp)\to 1$ as $a\to\infty$, and then (\ref{ac-pc-eq2})
implies the last claim. 
\end{proof}

\begin{lemma} \label{ac-pcstar} The critical value $a_c$ satisfies 
the following strict inequality:
$$\exp[-2a_c^\ell/\ell!] < 1/2.$$
In particular, $a_c>0$. 
\end{lemma}
\begin{proof} 
Pick an $\alpha>0$.
Given a configuration $\xi_0=\xi_0^{(a,\epsilon)}$, 
declare a site 
$x$ \emph{red} if $\xi_0(x)=3$, or $\xi_0(x)=2$ and the 
local configuration in $\xi_0$ around $x$ is:
\begin{equation}\label{ac-pcstar-eq0}
\begin{matrix}
3 \,3\, *\\
3\, 2 \,3\\
* \,3 \,3\\
\end{matrix} \qquad
\end{equation}
where $*$ denotes an arbitrary state.

The triangular lattice $\bT$ is obtained by adding 
SW-NE edges to the nearest neighbor edges in $\bZ^2$. 
(When we say that $x,y\in \bZ^2$ are neighbors without 
specifying the lattice, 
we still mean nearest neighbors.)
Recall that $\bT$ is (site-)self-dual and so the 
site percolation on $\bT$ has critical density $1/2$.
We call a $\bT$-circuit $\zeta$ a sequence of 
distinct points $y_0, y_1,\ldots, y_n=y_0$ such 
that $y_i$ and $y_{i-1}$ are $\bT$-neighbors 
for $i=1,\ldots, n$. We will also assume that 
$\zeta$ is a boundary of its connected interior, i.e., 
its sites are all points, which are outside some nonempty 
$\bT$-connected set $S$, but have a $\bT$-neighbor in $S$ (this is possible, again, because $\bT$ is site-self-dual); we call 
$S$ the \emph{interior} of $\zeta$.
Observe that every site on $\zeta$ has at least 
two neighbors in the set obtained as the union 
of sites on $\zeta$ and its interior.

Let $\redcircuit_N$ 
be the event that there exists a
$\bT$-circuit of red sites, with the origin in its interior, and inside 
$[-N,N]^2$. Moreover, let $\nozero_N$ be the event that 
no site $x\in[-N,N]^2$ has $\xi_0(x)=0$.
It follows from \cite{AG,BBR}, and standard arguments
from percolation theory, that there exists an $a$ with
$\exp[-2a^\ell/\ell!]<1/2$, with the following 
property. For every $\alpha>0$, there exists an $N=N(\alpha)$ 
so that
\begin{equation}\label{ac-pcstar-eq1}
\P(\redcircuit_N)>1-\alpha. 
\end{equation}

Pick any $\bT$-circuit $\zeta$ of red states.
Form the set of sites $R$ that consists of: all sites 
of $\zeta$; all sites in the interior of $\zeta$; and all
sites required to be in $\xi_0$-state $3$ 
in (\ref{ac-pcstar-eq0}) around any 
site with state $2$ on $\zeta$. Assume that there is 
no site in $\xi_0$-state $0$ in $R$. Then we claim that 
no site in $R$ ever changes its state to $0$. 
Indeed, to get a
contradiction, let $x\in R$ be the first such site to change its state to $0$ (chosen arbitrarily in 
case of a tie). Clearly $x$ cannot be in the interior of 
$\zeta$, as then $x$ has no neighbor outside $R$. The site
$x$ cannot have $\xi_0$-state 3 and be
on $\zeta$, as $x$ then has at least two neighbors 
in $R$, and hence at most two outside $R$. Furthermore, $x$ cannot be a site with $\xi_0$-state $2$ on $\zeta$, as $x$ must then have all neighbors in $R$ in accordance with $(\ref{ac-pcstar-eq0})$. The final 
possibility is that $x$ is one of the sites with $\xi_0$-state 
$3$ in  (\ref{ac-pcstar-eq0}). But each of those sites clearly also has two neighbors in $R$. 

So we have, for every $N$,
\begin{equation}\label{ac-pcstar-eq2}
\redcircuit_N\cap\nozero_N\subset \{\xi_\infty(\mathbf{0})=0\}^c.
\end{equation}
It follows from (\ref{ac-pcstar-eq1}) and (\ref{ac-pcstar-eq2})
that there exists an $N=N(\alpha)$ so that 
\begin{equation}\label{ac-pcstar-eq3}
\P(\xi_\infty(\mathbf{0})=0)\le \alpha +(2N+1)^2\epsilon.
\end{equation}
Now  in (\ref{ac-pcstar-eq3}), we send $\epsilon\to 0$ first, and then send $\alpha\to 0$ 
to conclude that $\P(\xi_\infty(\mathbf{0})=0)\to 0$ as $\epsilon\to 0$, and therefore $a\le a_c$. 
\end{proof}


\subsection{The supercritical regime for $\ell\ge 2$}

\begin{lemma} \label{easy-coupling} 
Assume $\vec X=(X_1,X_2,X_3,X_4)$ 
and $\vec Y=(Y_1,Y_2,Y_3,Y_4)$
are $4$-tuples of i.i.d. Bernoulli random variables with 
$\P(X_i=1)=\alpha_1$ and $P(Y_i=1)=\alpha_2$ for all $i$. If 
$1-(1-\alpha_1)^4\le \alpha_2^4$, then $\vec X$ and $\vec Y$ can be coupled 
so that $\{\exists i: X_i=1\}\subset \{\forall i: Y_i=1\}$.
\end{lemma}
\begin{proof}
This follows from an elementary argument and we omit the details.
\end{proof}

\begin{lemma} \label{ac-super} If $a>a_c$, then  
(\ref{odd-theta-super}) holds.  Moreover, (\ref{odd-theta-super1}) holds.  
\end{lemma}
\begin{proof}  
Fix an $a'\in (a_c,a)$. 
Fix also a small $\delta>0$, to be chosen later dependent on 
$a'$. 
For $i=0,\ldots,5$, we define probabilities $\pn_i$ as follows. For $i=1,2,3,4$, let 
$$\pn_i=\probsub{p}{K_n^2
\text{ is $(\theta-i)$-IS but not $(\theta-i+1)$-IS}},
$$
and 
$$\pn_0=\probsub{p}{K_n^2
\text{ is $\theta$-IS}}, \quad \pn_5=\probsub{p}{K_n^2
\text{ is not $(\theta-4)$-IS}}.
$$
Denote by $\pi(\alpha)$ the Bernoulli product measure of \emph{active} and \emph{inactive} sites with density 
$\alpha$ of active sites. Build the initial state $\bxi_0$ 
in four steps as follows.
In the first step, choose active sites according to 
$\pi(\pn_4+\pn_5)$ and fill them with $5$s. In the second step,
choose active sites according to $\pi(\pn_0/(1-\pn_4-\pn_5))$ 
and fill them with $0$s, provided they are not already filled.
Continue in the third step with 
$\pi(\pn_3/(1-\pn_0-\pn_4-\pn_5))$ to fill some unfilled
sites with $3$s, and then in the fourth step 
analogously with $2$s, and 
then finally $1$s fill all the remaining unfilled sites. 

Divide $\bZ^2$ into $2\times 2$ boxes and couple product 
measures 
$\pi(\pn_4+\pn_5)$ and $\pi(\delta)$ 
on the space of pairs 
$(\eta_1,\eta_2)\in 2^{\bZ^2}\times 2^{\bZ^2}$ so that 
any box with at least one active site in $\eta_1$ is fully activated in $\eta_2$. 
This coupling is possible, for large enough $n$, by Lemmas~\ref{internal-results-odd} and~\ref{easy-coupling}. 

Use this to couple $\bxi_0$ with 
another initial state $\hxi_0$. To build this configuration, keep 
all selected product measures used to define $\bxi_0$, 
but change 
the first step above as follows:
replace $\pi(\pn_4+\pn_5)$ by $\pi(\delta)$ 
(coupled as above), and fill the active sites by $3$s
(instead of $5$s). Note that we now fill by $3$s twice, and 
that some $0$s, $1$s, and $2$s in $\bxi_0$ are 
converted to $3$s in 
$\hxi_0$. 

Denote the 
resulting bootstrap dynamics by $\bxi_t$ and $\hxi_t$. 
The important observation is that no site that is $5$ 
in $\bxi_0$ can ever turn to $0$ in $\hxi_t$, as it is covered
by a $2\times 2$ block of $3$s that cannot change. 
Therefore, by Lemma~\ref{omega lower bound} and the coupling between $\bxi_t$ and $\hxi_t$, 
\begin{equation}\label{acsuper-eq1}
\probsub{p}{\{\mathbf{0}\}\times K_n^2 \subset \omega_\infty} 
\ge \P(\bxi_\infty(\mathbf{0})=0)\ge \P(\hxi_\infty(\mathbf{0})=0).
\end{equation}
Now if $\delta=\delta(a')$ is small enough, then for
large enough $n$, 
\begin{equation}\label{acsuper-eq2}
\begin{aligned}
&\epsilon_n=\P(\hxi_0(\mathbf{0})=0)>0,\\
&\P(\hxi_0(\mathbf{0})=1)\ge \P(\xi^{(a',\epsilon_n)}(\mathbf{0})=1),\\
&\P(\hxi_0(\mathbf{0})=3)\le \P(\xi^{(a',\epsilon_n)}(\mathbf{0})=3).
\end{aligned}
\end{equation}
As $a'>a_c$, the inequalities (\ref{acsuper-eq2}) guarantee 
that $\liminf_n\P(\hxi_\infty(\mathbf{0})=0)>0$. 
Therefore, by (\ref{acsuper-eq1}), the leftmost inequality 
in (\ref{odd-theta-super}) holds. When $a\to\infty$, we can 
send $a'\to\infty$ as well, and then Lemma~\ref{ac-pc}
gives (\ref{odd-theta-super1}). 

To prove the rightmost inequality 
in (\ref{odd-theta-super})
let $\obstbox$ be the event that 
$\{x\}\times K_n^2$ is $(\theta-2)$-inert for all 
$x \in \{\mathbf{0},(0,1),(1,0),(1,1)\}$. Then 
$$\obstbox\subset \{\omega_\infty=\omega_0\text{ on }\{\mathbf{0}\}\times K_n^2\},$$ 
and therefore, for any $a>0$, by Lemmas~\ref{internal-results-odd} and~\ref{inertness-odd},
$$
\limsup_{n\to\infty} \probsub{p}{v_0 \in \omega_\infty}\le \lim_{n\to\infty}\probsub{p}{\obstbox^c}=1-\exp(-8a^\ell/\ell!)<1,
$$
which ends the proof of (\ref{odd-theta-super}).
\end{proof}

\subsection{The subcritical regime for $\ell\ge 2$}

\begin{lemma} \label{ac-sub} Assume that $a<a_c$ and $\ell\ge 2$. 
Then (\ref{odd-theta-sub}) holds. 
\end{lemma}
\begin{proof}
Pick now an $a'\in (a,a_c)$ and $\alpha>0$, and again also fix 
$\delta>0$, to be chosen later to be appropriately dependent 
on $a'$ and $\alpha$. We will redefine $\pn_i$, $\bxi_0$ and $\hxi_0$ 
from the previous proof. 

Let 
\begin{equation*}
\begin{aligned}
&\pn_0=\probsub{p}{K_n^2
\text{ is not $\theta$-II}}\\
&\pn_1=\probsub{p}{K_n^2
\text{ is not $(\theta-1)$-II but is $\theta$-II}},\\
&\pn_2=\probsub{p}{K_n^2
\text{ is not $(\theta-2)$-II but is $(\theta-1)$-II}},\\
&\pn_3=\probsub{p}{K_n^2
\text{ is $(\theta-2)$-II}}.\\
\end{aligned}
\end{equation*}
This time build the initial state $\bxi_0$ 
in three steps as follows.
In the first step, choose active sites according to 
$\pi(\pn_3)$ and fill them by $3$s. In the second step,
choose active sites according to $\pi(\pn_2/
(1-\pn_3))$ 
and fill them by $2$s, provided they are not already filled.
In the third step, choose the 
configuration of \emph{bad} sites: those are sites that
\begin{itemize}
\item are not $\theta$-II; or 
\item are internally inert but not inert for some threshold 
in $[\theta-2,\theta]$.
\end{itemize} 
The configuration of bad sites has proper conditional distribution 
given the configuration of $3$s and $2$s. Observe that 
this conditional distribution has finite range of dependence: 
if $||x-y||_1\ge 3$, then $x$ and $y$ are bad independently. 
Furthermore, by Lemma~\ref{inertness-odd}, the probability that any site is bad is, 
uniformly over the configurations of $2$s and $3$s,
$n^{-1+1/\ell+o(1)}$  and thus goes to $0$ if $\ell\ge 2$.
Finally, finish the construction of $\bxi_0$ by filling all
bad sites with $0$'s and the remaining unfilled sites 
with $1$s.

By \cite{LSS}, the configuration of bad sites can be coupled
with a product measure $\pi(\delta)$ that dominates it, and 
is independent of the configuration of $2$s and $3$s. 
As in the previous proof, we now couple $\bxi_0$ with 
another initial state $\hxi_0$. To build $\hxi_0$, keep 
the selected product measures used in the first two steps.
The third step is changed by using the $\pi(\delta)$, 
obtained from the domination coupling, as active sites, all of which are filled by $0$s, possibly replacing some $2$s and $3s$. This way, some of the $1$s, $2$s and $3$s in $\bxi_0$ are changed to 
$0$s in $\hxi_0$.  

Denote again the 
resulting bootstrap dynamics by $\bxi_t$ and $\hxi_t$.  
This time, by Lemma~\ref{omega upper bound} and coupling properties, 
\begin{equation}\label{acsub-eq1}
\probsub{p}{\omega_\infty\ne \omega_0\text{ on }\{\mathbf{0}\}\times K_n^2} 
\le \P(\bxi_\infty(\mathbf{0})=0)\le \P(\hxi_\infty(\mathbf{0})=0).
\end{equation}
Now if $\delta=\delta(a')$ is small enough, then for
large enough $n$, 
\begin{equation}\label{acsub-eq2}
\begin{aligned}
&\P(\hxi_0(\mathbf{0})=0)\le \delta,\\
&\P(\hxi_0(\mathbf{0})=1)\le \P(\xi^{(a',\epsilon)}(\mathbf{0})=1),\\
&\P(\hxi_0(\mathbf{0})=3)\ge \P(\xi^{(a',\epsilon)}(\mathbf{0})=3).
\end{aligned}
\end{equation}
As $a'<a_c$, the inequalities (\ref{acsub-eq2}) guarantee 
that $\P(\hxi_\infty(\mathbf{0})=0)<\alpha$ if $\delta=\delta(a',\alpha)$ is small enough. 
Therefore, by (\ref{acsub-eq1}), (\ref{odd-theta-sub}) holds.
\end{proof}

\subsection{The exceptional case: $\theta=3$}

We assume here that $p=a/n^2$, in accordance with (\ref{form-p-odd}). In this case, we need another version of the heterogeneous 
bootstrap dynamics, somewhere between $\xi_t$ used 
when $\ell\ge 2$ and $\zeta_t$ used later for the graph 
$\bZ^2\times K_n$. Indeed, observe that the obstacles 
are now empty Hamming planes, but they become completely 
occupied by contact with two fully occupied neighboring 
planes and another neighboring plane that is merely non-empty. 
Clearly, the probability of having a non-empty neighboring plane 
does not go to $0$, and so this possibility now cannot be 
handled by a coupling with a low-density measure. 

We denote the new rule by $\chi_t\in \{0,1,2,3\}^{\bZ^2}$, $t\in \bZ_+$. Assume that $\chi_0$ is given. For a given $t\ge 0$, let as before $Z_t(x)$ be the cardinality of $\{y: y\sim x\text{ and }\chi_t(y)=0\}$ and let $W_t(x) = \ind{\{y: y\sim x\text{ and }0<\chi_t(y)<3\} \neq \emptyset}$
then  
$$
\chi_{t+1}(x)=
\begin{cases}
0 & Z_t(x)\ge \chi_t(x) \text{ or } 
(\chi_t(x)=3,\, Z_t(x)=2, \text{ and }W_t(x)=1)
\\
\chi_t(x) &\text{otherwise.}
\end{cases}
$$
For a small $\epsilon>0$, we 
consider the initial state $\chi^{(a,\epsilon)}_0$ given by the 
product measure with 
\begin{equation*}
\begin{aligned}
\P(\chi^{(a,\epsilon)}_0(x)=0)&=\epsilon,\\
\P(\chi^{(a,\epsilon)}_0(x)=1)&=1-(a+1)e^{-a}, \\
\P(\chi^{(a,\epsilon)}_0(x)=3)&=e^{-a},\\
\P(\chi^{(a,\epsilon)}_0(x)=2)&=1-P(\chi^{(a,\epsilon)}_0(x)=0)-P(\chi^{(a,\epsilon)}_0(x)=1)-P(\chi^{(a,\epsilon)}_0(x)=3)
\end{aligned} 
\end{equation*}
for every $x\in\bZ^2$, denote the resulting bootstrap dynamics 
by $\chi_t^{(a,\epsilon)}$, and for $\theta=3$ define 
$a_c\in [0,\infty]$ by
$$
a_c=\inf\{a> 0: \lim_{\epsilon\to 0}\P(\chi^{(a,\epsilon)}_\infty(\mathbf{0})=0)>0\}.
$$

We will not provide complete proofs of the next three
lemmas, but only point to previous arguments that 
apply with simplifications and minor modifications.

\begin{lemma} \label{theta-3-ac-pc} The following strict inequalities hold:
 $$1-(a_c+1)e^{-a_c}<p_c^{\text{\rm site}},\quad e^{-a_c}<p_c^{\text{\rm site}}.$$  
 In particular, $a_c\in(0,\infty)$. Also, 
 $\lim_{\epsilon\to 0}\P(\chi^{(a,\epsilon)}_\infty(\mathbf{0})=0)
 \to 1$ as $a\to\infty$.
 
\end{lemma}

\begin{proof}
The argument is very similar to that for Lemmas~\ref{ac-pc}
and~\ref{ac-pcstar}. 
\end{proof}

\begin{lemma} \label{theta-3-ac-super}
If $a>a_c$, then (\ref{odd-theta-super}) holds.
Also, (\ref{odd-theta-super1}) holds. 
\end{lemma}

\begin{proof}
This follows from the proof of Lemma~\ref{ac-super}, 
simplified by the absence of states $4$ and $5$, which 
eliminates the need for a coupling domination. 
\end{proof}

\begin{lemma} \label{theta-3-ac-sub}
Assume that $a<a_c$. Then (\ref{odd-theta-sub}) holds.
\end{lemma}

\begin{proof}
The difference from the proof of Lemma~\ref{ac-sub}
is the definition of \emph{bad} sites, which in this 
case are those that are not $3$-inert, and those that are 
 $2$-II but not $2$-inert. As 
the density of bad sites goes to $0$ by 
Lemma~\ref{inertness-odd}, 
the proof of Lemma~\ref{ac-sub} can be easily adapted. 
\end{proof}

\section{Bootstrap percolation on $\bZ^2\times K_n$}

\newcommand{\nz}{\text{\tt nz}}
\newcommand{\zetaf}{\zeta^{\tt fv}}
\newcommand{\zetar}{\zeta^{\tt rs}}


In this section, we prove Theorem~\ref{z2xK theorem}, 
which follows from Lemmas~\ref{Z2xK limit} and ~\ref{Z2xK continuity} below. 

As already announced, we need yet another 
heterogeneous bootstrap rule in which 
sites in $\bZ^2$ receive more help from their neighbors than 
in $\xi_t$. In this case we have a new state, 
labeled by $\theta$ and representing 
an empty site that has no contribution to make. We 
denote this rule by $\zeta_t\in \{0,1,2,3,4,5, \theta\}^{\bZ^2}$, $t\in \bZ_+$. Assume 
that $\zeta_0$ is given. For a given $t\ge 0$, let as before $Z_t(x)$ be the cardinality of $\{y: y\sim x\text{ and }\zeta_t(y)=0\}$ and $W_t(x) = \ind{\{y: y\sim x\text{ and }0<\zeta_t(y)<\theta\} \neq \emptyset}$
then  
$$
\zeta_{t+1}(x)=
\begin{cases}
0 & 
Z_t(x)+W_t(x)\ge \zeta_t(x)\\
\zeta_t(x) &\text{otherwise.}
\end{cases}
$$

For an initially occupied set $\omega_0$, we create 
two initial states $\zeta_0$ as follows. For $x\in \bZ^2$, let $$N_x=|\{y\in \{x\}\times K_n: \omega_0(y)=1\}|.$$
Call $x$ a \emph{clash site}
if $N_x<\theta$ and $\omega_0(y_1,u)=\omega_0(y_2,u)=1$ for some $y_1\ne y_2$ 
in $\{x\}\cup\{y:y\sim x\}$ and some $u\in K_n$, such that $N_{y_1}<\theta$ and
$N_{y_2}<\theta$.
We define the \emph{favoring initialization} $\zetaf_0(x)$ and the \emph{restricting initialization} $\zetar_0(x)$ as follows.
If $x$ is a clash site, then $\zetaf_0(x)=0$, while $\zetar_0(x)=\theta$. If $x$ is not a clash site, 
the two initializations are equal: $\zetaf_0(x)=\zetar_0(x)=\nz(N_x)$, where $\nz:\bZ_+\to\{0,\ldots,5,\theta\}$ is given by 
\begin{equation}\label{N-zeta}
\nz(m)=
\begin{cases}
0&m\ge \theta\\
k&m=\theta-k\text{ for some }k\in\{1,2,3,4\}\\
5&0<m<\theta-4\\
\theta &m=0
\end{cases}
\end{equation} 
These initializations determine their respective dynamics 
$\zetar_t$ and $\zetaf_t$, $0\le t\le \infty$.
We next state the comparison lemma whose 
simple proof is omitted. 

\begin{lemma}\label{Z2xK bounds}
We have
$$
\bigcup\{\{x\}\times K_n : \zetar_\infty(x)=0\}\subset\omega_\infty\subset\bigcup\{\{x\}\times K_n: \zetaf_\infty(x)=0\}\cup \omega_0.
$$
\end{lemma}

Consider $\bZ^2\times [0,\infty)$ and equip each $\{x\}\times [0,\infty)$, $x\in \bZ^2$ with an independent Poisson point location of unit intensity. Then we define the \emph{$a$-initialization} 
$\zeta_0^{(a)}$ obtained by $\zeta_0^{(a)}(x)=\nz(N_x^a)$, 
where now
$N_x^a$ is the number of location points in $\{x\}\times [0,a]$
and the function $\nz$ is defined in
(\ref{N-zeta}). 

For the rest of this section, we 
assume that $\omega_0$ is a product measure with density 
$p=a/n$.

\begin{lemma}\label{Z2xK coupling}
Assume $a'>a$. Then, for large enough $n$, $\omega_0$ and 
the $a'$-initialization $\zeta_0^{(a')}$
can be coupled so that $\zetaf_0\ge \zeta_0^{(a')}$

Conversely, assume $a'<a$. Then, for large enough $n$, 
 $\omega_0$ and    
$\zeta_0^{(a')}$
can be coupled so that $\zetar_0\le \zeta_0^{(a')}$.
\end{lemma}

\begin{proof}
We will prove only the first statement; the second is 
proved similarly. Observe that the random variables $N_x$ are 
i.i.d.~Binomial($n$,$p$). Fix an $\epsilon>0$
such that $a+\epsilon<a'$. 

Assume that first the i.i.d.~random field of truncated 
random variables $N_x\wedge\theta$, $x\in \bZ^2$, 
is selected. Conditional on this selection, any site 
$x\in\bZ^2$ is a clash site with probability at most 
$C/n$, where $C = C(\theta)$ is a constant. Furthermore, 
if $||x-x'||_1\ge 3$, then $x$ and $x'$ are clash sites independently. Therefore, by \cite{LSS}, there exists  
an i.i.d random field $\eta_x$, $x\in \bZ^2$ 
of Bernoulli random variables, 
independent also of the field $N_x\wedge\theta$, $x\in \bZ^2$, so that 
$\eta_x=1$ whenever $x$ is a clash site and $\prob{\eta_x=1}=\epsilon$. 

If $n$ is large enough, we can, for a fixed $x$, 
find a coupling between 
$(N_x,\eta_x)$ and a Poisson($a$) random variable 
$M_x$ so that $(N_x\wedge \theta)\ind{\eta_x=0}\ge (M_x\wedge \theta)$. Thus we can construct
an independent field $M_x, x\in \bZ^2$ with this property, which concludes
the proof.
\end{proof}
  
Define now
\begin{equation}\label{phi-def}
\phi(a)=\P(\zeta^{(a)}_\infty(\mathbf{0})=0).
\end{equation}
Observe that $\phi:(0,\infty)\to[0,1]$ is a 
nondecreasing limit of nondecreasing 
continuous functions $\phi_t$ given by 
$
\phi_t(a)=\P(\zeta^{(a)}_t(\mathbf{0})=0).
$
Therefore, $\phi$ is left-continuous and nondecreasing. 

\begin{lemma}\label{Z2xK limit} Assume $\theta\ge 3$. 
Fix any $a\in (0,\infty)$ and $v\in \bZ^2\times K_n$. 
As $n\to\infty$, 
\begin{equation*}
\begin{aligned}
\P(\text{\rm Poisson}(a)\ge \theta)\le \phi(a)&\le \liminf_n\probsub{p}{\omega_\infty(v)=1}\\
&\le \limsup_n\probsub{p}{\omega_\infty(v)=1}\le \phi(a+)
\le 1-e^{-4a}
\end{aligned}
\end{equation*}
\end{lemma}

\begin{proof}
We have, 
$$
\{N_{\mathbf{0}}^a\ge \theta\}=\{\zeta^{(a)}_0(\mathbf{0})=0\}\subset 
\{\zeta^{(a)}_\infty(\mathbf{0})=0\},$$
and, for any $2\times 2$ block $B\subset\bZ^2$ including 
$\mathbf{0}$,
$$
\cap_{x\in B}\{N_x^a=0\}=\cap_{x\in B}\{\zeta^{(a)}_0(x)=\theta\}\subset 
\{\zeta^{(a)}_\infty(\mathbf{0})=\theta\},$$
which gives the two extreme bounds. The remainder follows 
from Lemmas~\ref{Z2xK bounds}
and \ref{Z2xK coupling}. 
\end{proof}

\begin{lemma}\label{Z2xK continuity} For 
$\theta\ge 14$, $\phi$ is continuous on $(0,\infty)$.
\end{lemma}

\begin{proof}
Recall that by the construction, $\zeta_t^{(a)}$ are 
coupled for all $a$. Let 
$$
E_a=\bigcap_{a'>a} \{\zeta^{(a')}_\infty(\mathbf{0})=0\},
$$
so that $\phi(a+)=\P(E_a)$. Let also $F_a$ be the 
event that there is an $\ell^\infty$-circuit $\cC$ around the origin, 
consisting of sites $x$ with $N_x^a\notin [\theta-5,\theta-1]$. 
As  no site 
in $\cC$ ever changes its state in the $\zeta_t^{(a)}$ dynamics, 
$$
E_a\cap F_a\subset \{\zeta^{(a)}_\infty(\mathbf{0})=0\}.
$$
It remains to show that, for $\theta\ge 14$, 
 $\P(F_a)=1$ for all $a\in (0,\infty)$, that is, 
$$
\P(\text{\rm Poisson}(a)\in [\theta-5,\theta-1])\le 
p_c^{\text{\rm site}}.
$$
Using the rigorous lower bound $p_c^{\text{\rm site}}>0.556$ \cite{vdBE}, a numerical computation shows that the above bound
indeed holds for $\theta\ge 14$.
\end{proof}

\section{Open problems}

We conclude with a selection of a few natural questions. 

\begin{question} 
Is the function $\phi$ defined in (\ref{phi-def}) 
continuous on $(0,\infty)$ for all 
$\theta$? Is it analytic for all, or at least large enough, 
$\theta$? 
\end{question}

\begin{question} 
Is the function $a\mapsto\lim_{\epsilon\to 0}\P(\xi_\infty^{(a,\epsilon)}(\mathbf{0})=0)$, where 
$\xi_\infty^{(a,\epsilon)}$ is defined in Section 5.1, 
continuous for all $a$? A related question is whether $\lim_{n\to\infty} \P_p(v_0 \in \omega_\infty)$ exists for odd $\theta$ and all $a$ when $p$ is given by~\eqref{form-p-odd}?
\end{question}

In both question above, 
arguments similar to that for Lemma~\ref{Z2xK continuity} imply continuity for large enough $a$ and for small enough $a$.

\begin{question} 
When $a<a_c$ in Theorem~\ref{odd-theta-thm}, what is the rate 
of convergence in (\ref{odd-theta-sub})?
\end{question}

Our last three questions are more open-ended, and their 
answers likely
require development of new techniques.
We first propose a closer look into the critical 
scaling in Theorem~\ref{sharp transition even}.

\begin{question} Assume $\theta$ is even, as in Theorem~\ref{sharp transition even}. Assume that 
$$p=(2(\ell-1)!)^{1/\ell}\frac{(\log n)^{1/\ell}}{n^{1+1/\ell}}+bf(n)
$$
Can the function $f(n)$ be chosen so that the limit
of the final density as $n\to\infty$ exists and 
is neither a constant nor a step function of $b\in \bR$?
\end{question}

We conclude with two questions on larger dimensions 
of the lattice factor or the Hamming torus factor
(see also \cite{GHPS, GS}). 

\begin{question} What are the analogues of our main theorems
for bootstrap percolation on
$\bZ^d\times K_n^2$, for $d\ge 3$?
\end{question}

To approach this question using the methods of our present paper
would require a much deeper understanding of heterogeneous 
bootstrap percolation on $\bZ^d$ (see \cite{GHS}). 

\begin{question} What are the analogues of our main theorems
for bootstrap percolation on
$\bZ^2\times K_n^d$, $d\ge 3$?
\end{question}

This question poses a significant challenge 
at present, as the bootstrap percolation on $K_n^d$, $d\ge 3$, 
alone 
is poorly understood \cite{GHPS}, except for $\theta=2$
\cite{Sli}.

\section*{Acknowledgements}

JG  was  partially  supported  by  the  NSF  grant  DMS-1513340
and the Slovenian Research Agency research program P1-0285.
DS  was  partially  supported  by  the  NSF  TRIPODS  grant  
CCF-1740761.

\newcommand{\etalchar}[1]{$^{#1}$}

\end{document}